\documentclass[11pt,reqno]{amsart}
\usepackage{euscript,graphicx,overpic,epstopdf,
amscd,amsgen,amsfonts,amssymb,latexsym,
amsmath,amsthm,graphicx,mathrsfs,times}
\newtheorem{theorem}{Theorem}[section]
\newtheorem{lemma}[theorem]{Lemma}
\newtheorem{proposition}[theorem]{Proposition}

\newtheorem{corollary}[theorem]{Corollary}

\newcommand{\eqdef}{\stackrel{\scriptscriptstyle\rm def}{=}}
\DeclareMathOperator{\diam}{diam}

\DeclareMathOperator{\Per}{Per}

\DeclareMathOperator{\llangle}{\langle\hspace{-0.05cm}\langle}
\DeclareMathOperator{\rrangle}{\rangle\hspace{-0.05cm}\rangle}

\def\bR{\mathbb{R}}

\def\cK{\EuScript{K}}
\DeclareMathOperator{\Int}{Int}

\def\cP{\EuScript{P}}
\def\cR{\EuScript{R}}

\def\cT{\mathscr{T}}
\def\cO{\mathscr{O}}
\def\ccO{\EuScript{O}}
\def\cH{\EuScript{H}}

\def\cM{\EuScript{M}}
\def\cE{\mathcal{E}}
\def\cD{\mathcal{D}}
\def\cF{\mathcal{F}}
\def\cL{\EuScript{L}}

\DeclareMathSymbol{\varnothing}{\mathord}{AMSb}{"3F}
\renewcommand{\emptyset}{\varnothing}

\author[K.~Burns]{Keith Burns}
\address{Department of Mathematics, Northwestern University Evanston, IL 60208-2730, USA}
\email{burns@math.northwestern.edu}
\urladdr{}

\author[K.~Gelfert]{Katrin Gelfert} \address{Instituto de Matem\'atica, UFRJ,
Cidade Universit\'aria - Ilha do Fund\~ao, Rio de Janeiro 21945-909,  Brazil}
\email{gelfert@im.ufrj.br}

\begin{document}

\title[Thermodynamics for geodesic flows]{Thermodynamics for geodesic flows\\ of rank 1 surfaces}

\begin{abstract}
We investigate the spectrum of Lyapunov exponents for the geodesic flow of a compact rank 1 surface.
\end{abstract}

\begin{thanks}
{We would like to thank Todd Fisher, Micha\l\ Rams, and Rafael Ruggiero for discussions. KG has been supported by the Alexander von Humboldt Foundation Germany and by the CNPq Brazil. KB received support from NSF grant DMS-0701140 and Faperj Brazil.}
\end{thanks}

\maketitle

\section{Introduction}\label{s:1}
In this paper we consider a $C^\infty$ compact connected surface $M$ of negative Euler characteristic equipped with a Riemannian metric of nonpositive curvature. Let ${G}=\{g^t\}_{t\in\bR}$  be the geodesic flow on the unit tangent bundle $T^1M$. This flow is a natural and much studied example of nonuniform hyperbolicity. It preserves a natural smooth measure on $T^1M$ known as the Liouville measure. As we explain in Section~\ref{sec:prelim-g},
the geometry determines  two continuous one dimensional subbundles of $TT^1M$ that are invariant under the derivative of the flow $G$ and are everywhere transverse to the subbundle $F^0$ that is tangent to the orbit foliation. We denote these bundles by $F^u$ and $F^s$.

The {\em rank} of  a vector $v \in T^1M$ is the codimension of the space $F^u_v \oplus F^s_v$
 in $T_vT^1M$, which is three dimensional. The rank is $1$ on the {\em regular set} $\cR$ and $2$ on the {\em higher rank set} $\cH$. The set $\cR$ is obviously open and invariant.  It is also dense \cite{Bal:82}. A vector $v \in T^1M$ belongs to $\cR$ if and only if the geodesic $\gamma_v$ with $\dot\gamma_v(0) = v$ passes through a point at which the curvature is negative. The complementary set $\cH$ is closed, invariant and nowhere dense;  $v \in \cH$ if and only if the curvature at $\gamma_v(t)$ is $0$ for all $t$.  The Liouville measure of $\cH$ is $0$ in all known examples, but this has not been proved in general. The flow $G$ exhibits hyperbolic behavior on $\cR$ and
$F^u\oplus F^0 \oplus F^s$ is the hyperbolic splitting there. The vanishing of the curvature along the geodesics tangent to vectors in $\cH$ means that there is no hyperbolicity at all on $\cH$. The flow $G$ is Anosov if $\cH = \emptyset$.

Let $\chi$ denote the Lyapunov exponent associated to the subbundle $F^u$. Since $G$ preserves the Liouville measure,  it is easily seen that the Lyapunov exponent associated to $F^s$ must be $-\chi$ on the set $\cR$. Thus on $\cR$ the three exponents for the flow are $\chi$, $-\chi$ and the exponent $0$ associated to the bundle $F^0$ tangent to the flow direction. This is also true on $\cH$, because the lack of hyperbolicity means that all Lyapunov exponents are $0$ there. The Lyapunov exponents for $G$ are  therefore completely determined by the function $\chi$.

We study  the level sets of $\chi$. More precisely, we consider the sets
 $$
 \cL(\alpha) \eqdef \left\{ v \in T^1M\colon \text{$v$ is Lyapunov regular and $\chi(v) = \alpha$} \right\}.
 $$
The definition  of $\cL(\alpha)$ can be reformulated using the continuous function
\begin{equation}\label{phidef}
	\varphi^u(v) = -  \lim_{t\to0}\frac{1}{t}\log\,\lVert d{g}^t|_{ F^u_v}\rVert.
\end{equation}
The Lyapunov exponent $\chi$ is the Birkhoff average of $-\varphi^u$ and $\cL(\alpha)$ is the set where the forward and backward Birkhoff averages both exist and are both equal to $-\alpha$ (see also Section~\ref{sec:prelim-e} for further equivalent definitions).

For a set $Z \subset T^1M$, we denote by $\overline\chi(Z)$ and $\underline\chi(Z)$  the supremum and the infimum respectively of $\chi(v)$ as $v$ ranges over Lyapunov regular vectors in $Z$. If $Z = T^1M$,  we write simply $\overline\chi$ and $\underline\chi$.
We have $\underline\chi = 0$ if $\cH \neq \emptyset$ and $\underline\chi > 0$ otherwise.
We denote by $h(Z)$ the topological entropy of the flow $G$ on the set $Z$, and we write just $h$ if $Z = T^1M$.\footnote{Notice that the sets $\cL(\alpha)$ are in general non-compact and accordingly, we have to use the general concept of topological entropy of noncompact sets (see Section~\ref{sec:press-1} for details).}

\begin{proposition} \label{nonemptyprop}
$\cL(\alpha) \neq \emptyset$ for $\alpha \in [\underline\chi,\overline\chi]$.
\end{proposition}
Our main result gives  lower bounds for the Hausdorff dimension $\dim_{\rm H}\cL(\alpha)$ of $\cL(\alpha)$ and the entropy of $h(\cL(\alpha))$
of the geodesic flow on the set $\cL(\alpha)$ for $\alpha \in (\underline\chi,\overline\chi)$. We are primarily interested in the case when  $\cH \neq \emptyset$ and the flow $G$ is nonuniformly hyperbolic. Our result is a natural extension of what holds in the case when $\cH = \emptyset$ and the geodesic flow is Anosov (see Section~\ref{sec:3.1basicsets}).

Our estimates are obtained from  the topological pressure of multiples of the function $\varphi^u$. By the variational principle, the topological pressure $P(\varphi)$ of a continuous function $\varphi\colon T^1M \to \bR$  is the supremum over all measures $\mu$ that are invariant under the flow $G$ of
$$
 	P(\varphi,\mu) = h(\mu) + \int\varphi\,d\mu,
$$
 where $h(\mu)$ is the entropy of the time-$1$ map  $g^1$ with respect to the measure $\mu$. In particular
$$
  	P(q\varphi^u,\mu) = h(\mu) - q\,\chi(\mu),
$$
  where
$$
 	 \chi(\mu) = - \int \varphi^u\,d\mu.
$$
The function
$$
 	\cP\colon q \mapsto P(q\varphi^u)
$$
 is convex since it is the supremum of the linear functions $q \mapsto P(q\varphi^u,\mu)$ and  nonincreasing since $\varphi^u \leq 0$. Moreover, we have $\cP(0) = h$ and $\cP(1) = 0$.

 We show that the graph of $\cP$ has a supporting line $q \mapsto \cF_\alpha(q)$ with slope $-\alpha$ for each $\alpha \in [\underline\chi,\overline\chi]$. Let $\cD(\alpha)$ be the intercept of $\cF_\alpha$ with the horizontal axis and $\cE(\alpha)$  the intercept of $\cF_\alpha$ with the vertical axis. The function $\alpha\mapsto  - \cE(\alpha)$ is the convex conjugate of  
 $q\mapsto \cP(-q)$ under the Legendre-Fenchel transform; thus
  $$
  \cE(\alpha) = \inf_{q \in \bR}  \left( P(q\varphi^u) + q\alpha\right).
  $$
  Another characterization of $\cE$ is that $\cE(\alpha)$ is the maximum  of $h(\mu)$ for an invariant measure $\mu$ with $\chi(\mu) = \alpha$; see Lemma~\ref{Echarac}. Since $\cD(\alpha) = \cE(\alpha)/\alpha$, formulas involving $\cD$ are often expressed in terms of $\cE$.

\begin{theorem} \label{mainthm}
     For every $\alpha \in (\underline\chi,\overline\chi)$,
   the Hausdorff dimension of $\cL(\alpha)$ satisfies
    \begin{equation}\label{easy}
    	\dim_{\rm H}\cL(\alpha)
	\ge 1+2\,\cD(\alpha)
	= 1+2\,\frac{\cE(\alpha)}{\alpha},
    \end{equation}
    and the entropy on $\cL(\alpha)$ satisfies
    \begin{equation}\label{alsoeasy}
       	h(\cL(\alpha)) \geq \cE(\alpha).
   \end{equation}
\end{theorem}

Theorem~\ref{mainthm} extends  results of Pesin and Sadovskaya~\cite{PesSad:01} for conformal Axiom A flows and Barreira and Doutor \cite{BarDou:04} for compact locally maximal hyperbolic invariant sets on which the flow is conformal. In their situations the inequalities in the theorem are equalities. We believe that the same is true for Hausdorff dimension in our setting; this question will be studied in a sequel to this paper. In the  case of entropy, we apply a result of Bowen \cite{Bow:73b} to prove to following result.

\begin{theorem}\label{entupperboundthm}
  	$h(\cL(\alpha)) \leq \cE(\alpha)$  for $\alpha \in [\underline\chi,\overline\chi]$.
\end{theorem}

 It follows that $h(\cL(\alpha)) = \cE(\alpha)$  for $\alpha \in (\underline\chi,\overline\chi)$.

  Our approach  to Theorem~\ref{mainthm} is to ``exhaust'' the non-uniformly hyperbolic set $T^1M$ by a sequence of basic sets. A {\em basic set} is a compact locally maximal hyperbolic set on which the flow is transitive.   We prove:

  \begin{theorem}\label{cor:dense}
	There is a family of basic sets $\Lambda_1\subset\Lambda_2\subset\cdots\subset \cR$ such that
	\begin{equation}\label{e:convergenceexpo}
		\lim_{\ell\to\infty}\underline\chi(\Lambda_\ell)=\underline\chi
		\quad\text{ and }\quad
		\lim_{\ell\to\infty}\overline\chi(\Lambda_\ell)=\overline\chi.	
	\end{equation}	
This family can be chosen so that $\bigcup_\ell\Lambda_\ell$ is dense in $T^1M$
and for any basic set $\Lambda\subset\cR$, $\Lambda\ne\cR$, there exists $\ell\ge1$ such that $\Lambda\subset\Lambda_\ell$.
\end{theorem}
The main difficulty in proving Theorem~\ref{cor:dense} is to show that any closed hyperbolic subset of $\cR$  is contained in a locally maximal hyperbolic set. This is not true in general, as is shown by examples of Crovisier~\cite{Cro:02} and Fisher~\cite{Fis:06}.
We use an argument suggested by Anosov, which exploits the fact that a closed hyperbolic subset of $\cR$ is one dimensional (unless it is the whole of $\cR$).

 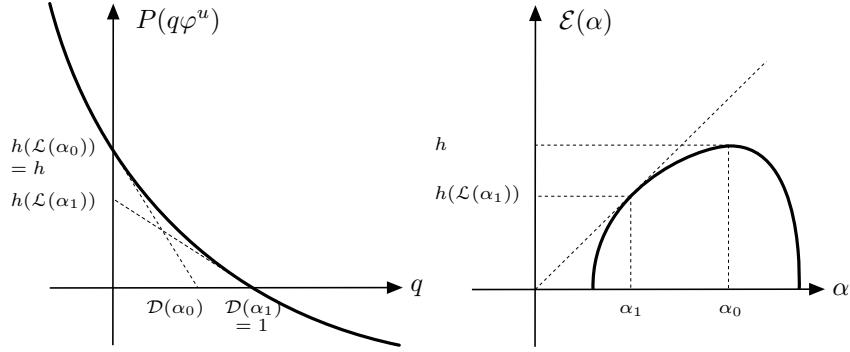
\begin{figure}
\begin{minipage}[c]{\linewidth}
\centering
\begin{overpic}[scale=.60
  ]{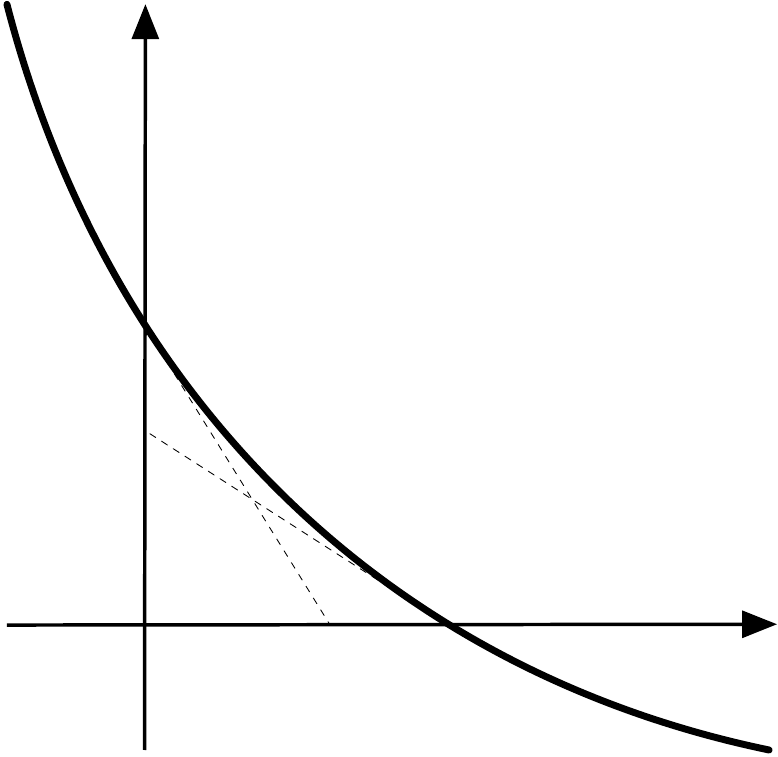}
      	\put(28,10){\tiny$\cD(\alpha_0)$}	
      	\put(50,10){\tiny$\cD(\alpha_1)$} \put(53,5){\tiny$=1$}	
      	\put(102,16){\small$q$}	
      	\put(25,90){\small$P(q\varphi^u)$}	
	\put(-10,55){\tiny$h({\cL(\alpha_0)})$}
	\put(-10,49){\tiny$=h$}
	\put(-10,40){\tiny$h({\cL(\alpha_1)})$}
\end{overpic}
\hspace{0.7cm}
\begin{overpic}[scale=.60
  ]{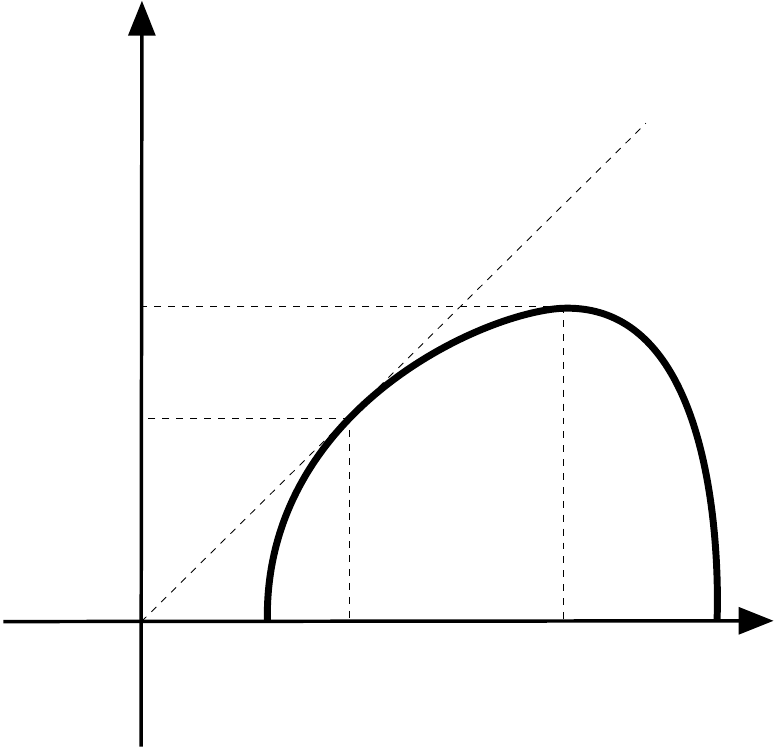}
      	\put(42,10){\tiny$\alpha_1$}
	\put(70,10){\tiny$\alpha_0$}	
      	\put(102,15){\small$\alpha$}	
      	\put(25,90){\small$\cE(\alpha)$}	
	\put(-10,56){\tiny$h$}
	\put(-10,42){\tiny$h({\cL(\alpha_1)})$}\end{overpic}
\end{minipage}
\caption{The pressure function $q\mapsto P(q\varphi^u)$ and its conjugate $\alpha\mapsto\cE(\alpha)$ and entropies and Hausdorff dimensions for the exponent $\alpha_0$ of the maximal entropy measure and the exponent $\alpha_1$ of the Liouville measure in the case that $G|_{T^1M}$ is Anosov.}
\label{fig.map6}
\end{figure}

If $\Lambda$ is basic, then $\dim_{\rm H}(\cL(\alpha) \cap \Lambda)$ and $h(\cL(\alpha) \cap \Lambda)$ are lower bounds for $\dim_{\rm H}\cL(\alpha)$ and $h(\cL(\alpha))$. One can  express $\dim_{\rm H}(\cL(\alpha) \cap \Lambda)$ and $h(\cL(\alpha) \cap \Lambda)$ using the functions $\cD_\Lambda$ and $\cE_\Lambda$, which are defined analogously to $\cD$ and $\cE$ starting from the function
   $$
   \cP_\Lambda\colon
   q \mapsto P(q\varphi^u,\Lambda).
   $$
   This function is real analytic and strictly convex.
  By results of Barreira and Doutor (see \cite[Corollary 15]{BarDou:04}),
  $$
    \text{$\dim_{\rm H}(\cL(\alpha) \cap \Lambda) =
    1 + 2\,\cD_\Lambda(\alpha)\quad$and\quad
    $h(\cL(\alpha) \cap \Lambda) = \cE_\Lambda(\alpha)$,}
    $$
for each $\alpha \in (\underline\chi(\Lambda),\overline\chi(\Lambda))$. Furthermore for each such $\alpha$ there is a unique ergodic measure $\mu$ supported on $\Lambda$ such that  $\chi(\mu) = \alpha$ and $q \mapsto P(q\varphi^u,\mu)$ is a supporting line for $\cP_\Lambda$.

   The key step in proving Theorem~\ref{mainthm} is then to show that $P(q\varphi^u,{\Lambda_\ell}) \to P(q\varphi^u)$ as $\ell \to \infty$ for each $q$ and that this convergence is uniform on compact subsets. This is accomplished using the analogue for flows in three dimensions of Katok's theorem that
 the entropy of a hyperbolic invariant measure for a surface diffeomorphism can be approximated by the entropy on a basic set (see~\cite{Kat:80,Kat:82,Kat:84} or~\cite[Supplement S.5]{KatHas:95}).

 Our methods also yield the following result, which may be of independent interest.

 \begin{theorem}\label{lem:lowe}
	The entropies $h(\Lambda)$ for the basic sets $\Lambda \subset \cR$ are dense in $[0,h]$.
\end{theorem}

 There are some distinctive features that arise only when $\cH \neq \emptyset$. Let $\widetilde m$  be the measure obtained by restricting the Liouville measure to the invariant open set $\cR$ and normalizing to obtain a probability measure. In all known examples $\widetilde m$ is just  the Liouville measure, but this has not been proved in general.\footnote{The issue is, of course, whether the Liouville measure of $\cH$ must be $0$. Our result in Theorem~\ref{Hnonzerothm} that $h(\cL(0))  = 0$ is consistent with the hope that the Liouville measure of $\cH$ is always $0$. }
 Ergodicity of $\widetilde m$ was proved in~\cite{Pes:77}.
 Set $\alpha_1 \eqdef \chi(\widetilde m)$.

 \begin{theorem}\label{Hnonzerothm}
 	If  $\,\cH \neq \emptyset$, we have $\cP(q) = 0$ for $q \geq 1$. Moreover
$\cD(\alpha) = 1$ and $\cE(\alpha) = \alpha$ for $\alpha \in [0, \alpha_1]$. Hence
 $\dim_{\rm H}\cL(\alpha) = 3$ for $\alpha \in (0,\alpha_1]$ and $h(\cL(\alpha)) = \alpha$ for $\alpha \in [0,\alpha_1]$. In particular $h(\cL(0))  = 0$.
 \end{theorem}

 We believe that  $\dim_{\rm H}\cL(0) = 3$ when $\cH \neq \emptyset$, but the methods of this paper  do not apply in that case.
 Since $\cH \subset \cL(0)$, Theorem~\ref{Hnonzerothm} gives $h(\cH) = 0$.  In higher dimensions, it is possible to have positive entropy on $\cH$; an example was given by Gromov~\cite{Gro:78}.

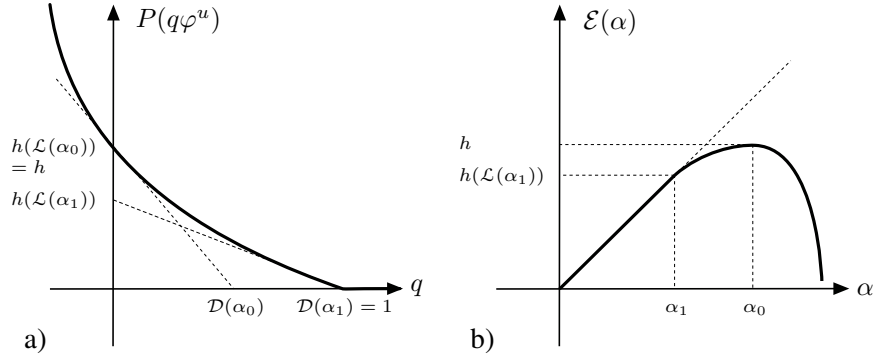
\begin{figure}
\begin{minipage}[c]{\linewidth}
\centering
a)\begin{overpic}[scale=.60
  ]{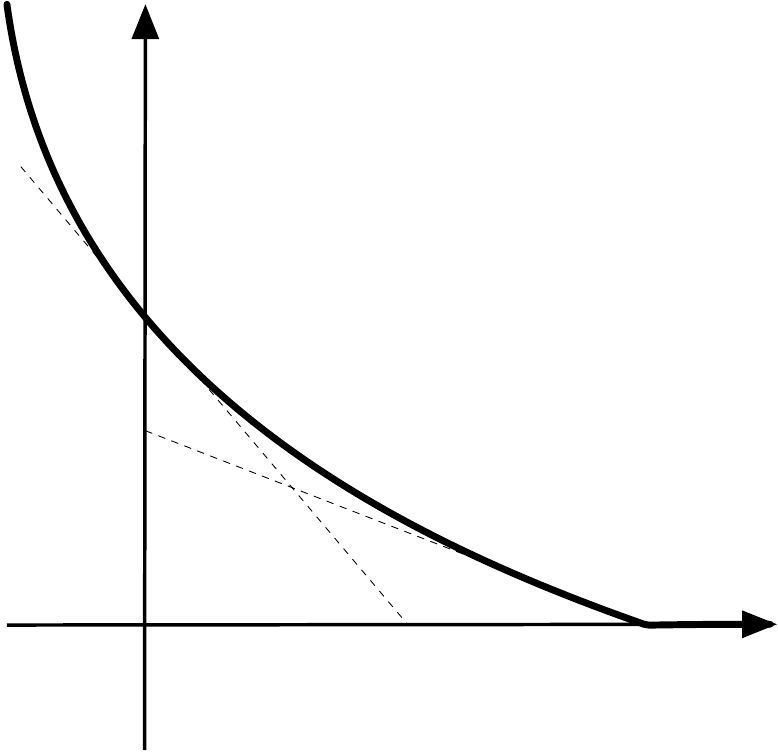}
      	\put(45,10){\tiny$\cD(\alpha_0)$}	
      	\put(70,10){\tiny$\cD(\alpha_1) = 1$} 
      	\put(102,16){\small$q$}	
      	\put(25,90){\small$P(q\varphi^u)$}	
	\put(-10,55){\tiny$h(\cL(\alpha_0))$}
	\put(-10,49){\tiny$=h$}
	\put(-10,40){\tiny$h(\cL(\alpha_1))$}
\end{overpic}
\hspace{0.7cm}
b)\begin{overpic}[scale=.60
  ]{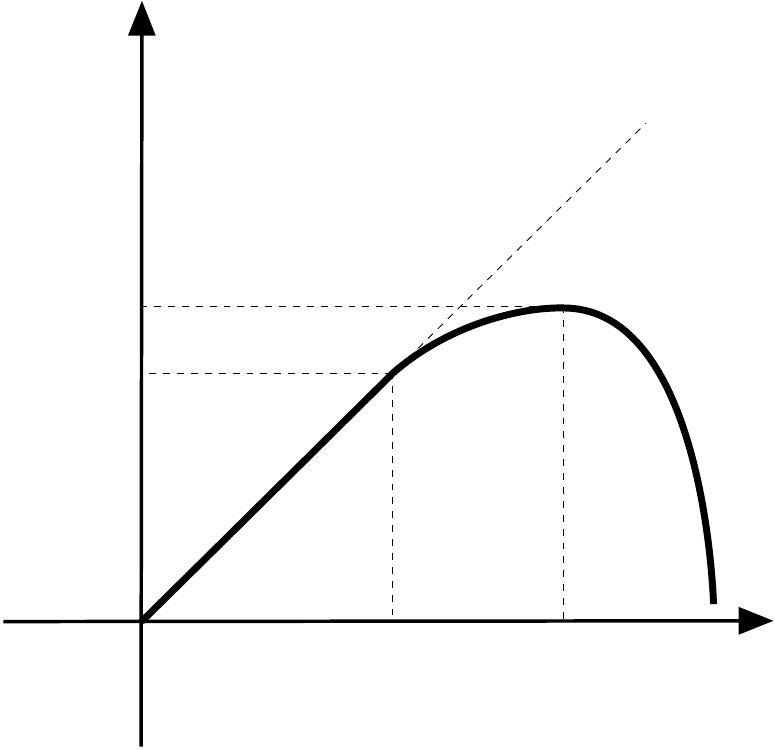}
      	\put(48,10){\tiny$\alpha_1$}
	\put(70,10){\tiny$\alpha_0$}	
      	\put(102,15){\small$\alpha$}	
      	\put(25,90){\small$\cE(\alpha)$}	
	\put(-10,56){\tiny$h$}
	\put(-10,47){\tiny$h(\cL(\alpha_1))$}\end{overpic}
\end{minipage}
\caption{As in Figure~\ref{fig.map6}, but in the case that $\cH\ne\emptyset$.}\label{fig.map5}
\end{figure}

 Figure~\ref{fig.map6}  shows the function $\cP\colon q \mapsto P(q\varphi^u)$ and its conjugate  $\alpha \mapsto \cE(\alpha)$ in the case when $\cH = \emptyset$ and  the flow $G$ is Anosov.
 The function $\cP$ is real analytic in this case. It is also strictly convex unless the curvature is constant, in which case it is linear (compare also Figure~\ref{fig.map33}).

 Figure~\ref{fig.map5} shows the same functions in the case when
 $\cH \neq \emptyset$. The results noted in Theorem~\ref{Hnonzerothm} are consequences of the corner that appears at $(1,0)$ in the graph of $\cP$ when $\cH \neq \emptyset$, as shown in Figure~\ref{fig.map5}.   We do not know if it is possible for
there to be other corners in the graph. Nor do we know if $\cP$ must
be strictly convex on $(-\infty,1)$ when $\cH \neq \emptyset$.

 In both figures, the graph of $\cP$ passes through $(0,h)$ and $(1,0)$. The function $\cP$ is differentiable at $q = 0$. The tangent line at $(0,h)$ is the graph of the function $q \mapsto
h - q\alpha_0$, where  $\alpha_0$ is the exponent for the measure of maximal entropy $\mu_{max}$.   Knieper~\cite{Kni:98} constructed this measure and showed that it is ergodic and the unique measure of maximal entropy.
The line $q \mapsto h(\widetilde m) - q\alpha_1$ passes through $(1,0)$ because the measure $\widetilde m$ is absolutely continuous with respect to the Liouville measure and thus $h(\widetilde m) = \chi(\widetilde m) = \alpha_1$ by Pesin's formula.

The organization of this paper is as follows. In Section~\ref{s:2} we collect some preliminary results on Lyapunov exponents and, in particular, discuss properties of the continuous invariant subbundle $F^u$ that gives rise to the \emph{continuous} potential $\varphi^u\colon T^1M\to\bR$.
In Section~\ref{s:3} we derive properties of the function $q\mapsto P(q\varphi^u)$ and of its conjugate $\alpha\mapsto \cE(\alpha)$. In particular, we prove Theorems~\ref{entupperboundthm} and~\ref{Hnonzerothm}.
In Section~\ref{s:4} we study uniformly hyperbolic subsystems.
In Section~\ref{s:5} we introduce a technique to bridge between uniformly hyperbolic Cantor sets and we provide all tools that are needed to prove Theorem~\ref{cor:dense}.
Section~\ref{s:6} collects spectral properties of uniformly hyperbolic sub-systems and contains the proofs of Proposition~\ref{nonemptyprop} and Theorems~\ref{mainthm} and~\ref{lem:lowe}.

\section{Preliminaries}\label{s:2}

\subsection{Geometry}\label{sec:prelim-g}

We consider a compact
 surface $M$ with a smooth Riemannian metric such that the curvature $K(p)$ is nonpositive for all $p \in M$.  Each vector $v\in TM$ determines a unique geodesic $\gamma_v(\cdot)$ such that $\dot\gamma_v(0)=v$. The geodesic flow $G=\{g^t\}_{t\in\bR}$ acts  by $g^t(v)=\dot\gamma_v(t)$. We study its restriction to the unit tangent tangent bundle $T^1M$, which is invariant under $G$.  A comprehensive reference for the material in this section is~\cite[Chaper 1]{Ebe:96}.

As usual, given $v\in T_pM$, we identify $T_vTM$ with  $T_pM\oplus T_pM$ via the isomorphism
\[
	\xi\mapsto (d\pi(\xi),C(\xi)),
\]
where $\pi: TM\to M$ denotes the canonical projection and $C\colon TTM\to TM$ denotes the connection map defined by the Levi Civita connection.
Under this isomorphism $T_vT^1M$ corresponds to $T_pM\oplus v^\perp$, where $v^\perp$ is the subspace of $T_pM$ orthogonal to $v$. The vector field that generates the geodesic flow is $V:v \mapsto (v,0)$.
The Riemannian metric on $M$ lifts to the {\em Sasaki metric} on $TM$ defined by
\[
 \llangle \xi,\eta\rrangle_v
 = \langle d\pi_v (\xi),d\pi _v(\eta)\rangle_{\pi(v)}
 	+ \langle C_v(\xi),C_v(\eta)\rangle_{\pi(v)}.
 \]

A Jacobi field $J$ along a geodesic $\gamma$ is a vector field along $\gamma$ that satisfies the Jacobi equation
\begin{equation}\label{e.j}
 J''(t) +R(J(t),\dot\gamma(t))\dot\gamma(t) = 0,
\end{equation}
where $R$ denotes the Riemannian curvature tensor of $M$ and $'$ denotes covariant differentiation along $\gamma$. It can be shown that if for some $t_0$ both $J(t_0)$ and $J'(t_0)$ are orthogonal to $\dot\gamma(t_0)$, then $J(t)$ and $J'(t)$ are orthogonal to $\dot\gamma(t)$ for all $t$. A Jacobi field with this property is said to be \emph{orthogonal}. If $J$ is an orthogonal Jacobi field, the Jacobi equation can be rewritten as
\begin{equation}\label{ej2}
 J''(t) + K(\gamma(t))J(t) = 0.
 \end{equation}
We can express an orthogonal Jacobi field as $J(t) = j(t)E(t)$, where $j$ is a scalar function and $E$ is a continuous unit vector field along $\gamma$ that is orthogonal to $\gamma$. It follows from \eqref{ej2} that
$u = j'/j$ satisfies the Riccati equation
 \begin{equation}\label{er}
 u'(t) + u(t)^2 +K(\gamma(t)) = 0.
  \end{equation}

Jacobi fields give a geometric description of the derivative of the geodesic flow.  Given $\xi\in T_v T^1M$, denote by $J_\xi$ the   Jacobi field along $\gamma_v$ with initial conditions $J_\xi(0)=d\pi_v(\xi)$ and $J_\xi'(0)=C_v(\xi)$. Then $dg^t_v(\xi)$  corresponds to $(J_\xi(t),J_\xi'(t))$. The orthogonal Jacobi fields correspond to an invariant  subbundle of $TT^1M$ whose fiber over $v$ is $v^\perp\oplus v^\perp$.

Nonpositivity of the curvature means that $\|J(t)\|$ is a convex function of $t$ for any Jacobi field $J$, because
 $$
 \frac12 \langle J,J \rangle'' =\langle J,J'\rangle'
	=\langle J',J'\rangle+\langle J'',J\rangle
	=\langle J',J'\rangle-\langle R(J,\dot\gamma_v),\dot\gamma_v,J\rangle) \geq 0.
$$
If $\|J\|$ is constant,  then $J'(t) = 0$ and $R(J(t),\dot\gamma(t)),\dot\gamma(t),J\rangle(t)) = 0$ for all $t$; in a surface, this means that the curvature vanishes everywhere along the geodesic.

We now describe the invariant subbundles $F^0$, $F^u$ and $F^s$, which were introduced in Section 1. The bundle $F^0$ is spanned by the vector field $V$ that generates the flow $G$. The other two bundles correspond to two special families of orthogonal Jacobi fields, the unstable Jacobi fields in the case of $F^u$ and the stable Jacobi fields in the case of $F^s$ . An orthogonal Jacobi field $J$ along a geodesic is called {\em unstable} (resp.~{\em stable}) if $\|J\|$ is nondecreasing (resp.~nonincreasing).
Since the length of a Jacobi  is a  convex function, it is easily seen that such Jacobi fields can be constructed as the limit as $T \to \infty$ of orthogonal Jacobi fields that vanish at $-T$ or at $T$. The one-dimensional distributions $F^u$ and $F^s$ obtained in this way are continuous
and invariant. The fact that elements of $F^u$ and $F^s$ correspond to orthogonal Jacobi fields ensures that $F^0$ is orthogonal to $F^u$ and $F^s$ in the Sasaki metric. 

Since the bundles $F^u$ and $F^s$ are both one dimensional,
the intersection $F^u_v \cap F^s_v$ is non-trivial if and only if $F^u_v = F^s_v$.
An orthogonal Jacobi field $J$ that is both stable and unstable has $\|J(t)\| \leq \|J(0)\|$ for all $t$, and it follows from the convexity of $\|J\|$ that $\|J(t)\|$ is constant. This means that $J'(t) = 0$ for all $t$ and the curvature vanishes everywhere along the geodesic. On the other hand, if the curvature does vanish everywhere along the geodesic, any  Jacobi field with initial derivative $0$ has constant length. It follows from these remarks that $F^u_v = F^s_v$  if and only if the curvature along $\gamma_v$ is everywhere zero.

The distributions $F^u$ and $F^s$ are integrable, as are $F^{0u} = F^0\oplus F^u$ and $F^{0s} = F^0\oplus F^s$. The leaves of the integral foliations are closely related to the {\em horocycles}. If $\widetilde v$ is a lift of $v \in T^1M$ to a unit vector tangent to the universal cover $\widetilde M$, the stable horosphere $H^s(\widetilde v)$ and the unstable horosphere $H^u(\widetilde v)$ are the limits as $r \to \infty$ of the circles of radius $r$ with centers at $\gamma_{\widetilde v}(r)$ and $\gamma_{\widetilde v}(-r)$.  The limiting procedure used above to construct $F^u$ and $F^s$ is the infinitesimal version of this construction. We denote by $W^s(\widetilde v)$ (resp.~$W^u(\widetilde v)$) the set of unit normals to $H^s(\widetilde v)$ (resp.~$H^u(\widetilde v)$) pointing to the same side as $\widetilde v$. The projections $W^s(v)$ and $W^u(v)$ to $T^1M$ of these sets are independent of the choice of lift of $v$. The stable leaves $W^s$ are tangent to $F^s$ and the unstable leaves $W^u$ are tangent to $F^u$.
Flowing these leaves produces the weak stable manifold $W^{0s}(v) = \bigcup_{t \in \bR} g^tW^s(v)$ and the weak unstable manifold $W^{0u}(v) = \bigcup_{t \in \bR} g^tW^u(v)$, which are the leaves of the integral foliations for $F^{0s}$ and $F^{0u}$. Eberlein and O'Neill  in~\cite{EbeNei:73} constructed a natural boundary for the universal cover $\widetilde M$, which is analogous to the sphere at infinity for  the Poincar\'e disc. A lift to  $T^1\widetilde M$ of a leaf of the foliation $W^{0s}$ (resp.~$W^{0u}$) consists of all vectors that point forwards (resp.~backwards) to a common point in this boundary.

The foliations $W^s$ and $W^u$ are minimal; see e.g.~Theorem 6.1 in \cite{Ebe:73trans}.

If $w$ and $w'$ are in the same leaf of the foliation $W^s$ and we choose compatible lifts $\widetilde w$ and $\widetilde w'$ to the universal cover, then the  distance between $\gamma_{\widetilde w}(t)$  and $\gamma_{\widetilde w'}(t)$ is a convex and nonincreasing function of $t$. 

\begin{lemma}\label{flatstriplemma}
	Suppose that $w$ and $w'$ are in the same leaf of the foliation $W^s$ and
$\Delta \eqdef  \lim_{t \to \infty} d(\gamma_{\widetilde w}(t),\gamma_{\widetilde w'}(t))> 0$. Then the geodesics $\gamma_w(t)$ and $\gamma_{w'}(t)$ converge to the edges of a flat strip of width $\Delta$ as $t \to \infty$.	
\end{lemma}

\begin{proof}
	We can choose a sequence $t_n \to \infty$ such that the vectors $g^{t_n}(w)$ and $g^{t_n}(w')$ converge to a pair of vectors $v$ and $v'$ in the same leaf of $W^s$ such that $d(\gamma_v(t),\gamma_{v'}(t)) = \Delta$ for all $t$. The flat strip theorem (\cite[Proposition 5.1]{EbeNei:73})
	says that these geodesics bound a flat strip (a totally geodesic and isometric immersion of the product of $\bR$ with an interval).
\end{proof}

We will use the following version of a Shadowing Lemma (see~\cite[Closing Lemma 4.5.15]{Ebe:96} and~\cite[Theorem 7.1]{CouSha:} for proofs in the particular case of shadowing a recurring geodesic).

\begin{lemma}[Shadowing Lemma]\label{lem:dosha}
 Given $\kappa > 0$, $\tau > 0$ and $\varepsilon > 0$, there is $\delta > 0$ with the following property.

	Let $p_0$, $p_{\pm 1}$, $\ldots \in M$ be points at which the curvature of $M$ is at most $-\kappa^2$. 	Suppose $v_0$, $v_{\pm 1}$, $\ldots\in T^1M$ are such that the footpoint of $v_j$ is in $B(p_j,\delta)$ for all $j$. Suppose  that there are times $\ldots<T_{-1}<T_0<T_1<\ldots$ such that $T_{j+1}-T_j\ge\tau$ and $d(\dot\gamma_{v_j}(T_{j+1}-T_j),v_{j+1})<\delta$ for all $j$. Then there is a geodesic $\gamma$ such that for all $j$ we have $d(\dot\gamma(t),\dot\gamma_{v_j}(t-T_j))<\varepsilon$ for $T_j\le t\le T_{j+1}$.

The geodesic $\gamma$ is unique up to reparametrization. It is periodic if the sequence of vectors $v_j$  is periodic.
\end{lemma}

Lemma~\ref{lem:dosha} is proved by iterated applications of the following  lemma.

 \begin{lemma}\label{lem:regneigh}
	Given $\kappa > 0$ and $\varepsilon > 0$, there is $\delta > 0$ with the following property.
	Let $p\in M$ be a point with $K(p) \leq -\kappa^2$. If $v$, $w\in T^1M$ have footpoints in $B(p,\delta)$ and satisfy $d(v,w)<\delta$, then there is a unit speed geodesic $\gamma$ such that $d(\dot\gamma(t),\dot\gamma_v(t))<\varepsilon$ and $d(\dot\gamma(-t),\dot\gamma_w(-t))<\varepsilon$ for all $t\ge 0$. The geodesic $\gamma$ is unique up to reparamerization.
\end{lemma}

\subsection{Ergodic theory}\label{sec:prelim-e}

As explained in the introduction, we study the Lyapunov exponent $\chi$ associated to the invariant subbundle $F^u$.  The forward Lyapunov exponent is the forward Birkhoff average of the function $\varphi^u$ defined in \eqref{phidef}:
 $$
 \chi_+(v) = \lim_{T \to \infty} - \frac1T\log\|dg^T|_{F^u_v}\|
               = \lim_{T \to \infty} - \frac1T \int_0^T \varphi^u(t)\,dt.
 $$
 The backward Lyapunov exponent is the backward Birkhoff average:
 $$
 \chi_-(v) = \lim_{T \to \infty} \frac1T\log\|dg^{-T}|_{F^u_v}\|
               =  \lim_{T \to \infty} - \frac1T \int_{-T}^0 \varphi^u(t)\,dt.
 $$
 The exponent $\chi(v)$ is defined when $ \chi_+(v)$ and $ \chi_-(v)$ both exist and are equal.

 The function $\varphi^u$ is continuous because the bundle $F^u$ is continuous. It vanishes on $\cH$ because unstable Jacobi fields are covariantly constant along geodesics tangent to vectors in $\cH$.

 \begin{lemma}\label{nonneglemma}
 	$\varphi^u \leq 0$.
 \end{lemma}

 \begin{proof}Let $J$ be an unstable Jacobi field along a geodesic $\gamma$. Choose a continuous field $E$ along $\gamma$ of unit vectors orthogonal to $\gamma$ and define the function $j$ by $J(t) = j(t)E(t)$. Then $jj' \geq 0$ because $J$ is an unstable Jacobi field and $\|J\|' = 2jj'$. Using this we obtain
 \begin{align*}
 \frac12\frac d{dt}\|(J(t),J'(t))\|^2 &=  \frac12\frac d{dt}\Big( \langle J(t),J(t) \rangle +  \langle J'(t),J'(t) \rangle\Big)\\
                    &= \langle J'(t),J(t) \rangle + \langle J''(t),J'(t) \rangle  \\
                    &= \langle J'(t),J(t) \rangle + \langle R(J(t),\dot\gamma)\dot\gamma,J'(t) \rangle  \\
                    &= j'(t)j(t) - K(\gamma(t))j(t)j'(t) \geq 0.
 \end{align*}
Hence  $\lVert(J(t),J'(t))\rVert$ is nondecreasing, and it follows that $\varphi^u \leq 0$.
 \end{proof}

 The following lemma gives an equivalent characterization of $\chi(\cdot)$.

\begin{lemma}
	For every $v\in T^1M$ Lyapunov regular we have
	\[
		\chi(v)
		= \lim_{t\to\infty}\frac 1 T \int_0^Tu(t)\,dt,
	\]
	where $u(t)=\lVert J_\xi'(t)\rVert/\lVert J_\xi(t)\rVert$, $\xi\in F_v^u$, satisfies $u'(t)=-u^2(t)-K(\gamma_v(t))$.
\end{lemma}

\begin{proof}
	By compactness, there exists $k$ such that the curvature is bounded by $-k^2\le K$. Hence, it follows from~\cite[Proposition 2.11]{Ebe:73b} that $\lVert J_\xi'(t)\rVert\le k\,\lVert J_\xi(t)\rVert$. Thus, $\lVert J_\xi(t)\rVert\le\lVert dg^t_v(\xi)\rVert\le \sqrt{1+k^2}\,\lVert J_\xi(t)\rVert$ and hence
	\[\begin{split}
		\lim_{T\to\infty}\frac 1 T \log\, \lVert dg^T|_{F^u_v}\rVert
		= \lim_{T\to\infty}\frac1T\log\frac{\lVert J_\xi(T)\rVert}{\lVert J_\xi(0)\rVert}
		=  \lim_{T\to\infty}\frac1T\int_0^T u(t)\,dt,
	\end{split}\]	
where $u(t)\eqdef \lVert J_\xi'(t)\rVert/\lVert J_\xi(t)\rVert$ satisfies $u'(t)=-u^2(t)-K(\gamma_v(t))$.
\end{proof}

 We denote by $\cM$ the set of all $G$-invariant probability measures on $T^1M$ and by $\cM_{\rm e}$ the subset of all ergodic measures in $\cM$.
 For a measure $\mu \in \cM$, we define
  $$
  \chi(\mu) = - \int \varphi^u\,d\mu.
  $$
  For a set $Z \subset T^1M$, we denote by $\cM(Z)$ (resp.~$\cM_{\rm e}(Z)$) the set of measures in $\cM$ (resp.~$\cM_{\rm e}$) that assign full measure to $Z$.
  If $\mu \in \cM_{\rm e}$, we have $\chi(v) = \chi(\mu)$ for $\mu$-a.e.~$v \in T^1M$ and thus $\mu \in \cM_{\rm e}(\cL(\chi(\mu)))$.

There are two measures in $\cM_{\rm{e}}$ of particular significance. The first is the unique measure of maximal entropy $\mu_{\rm max}$. This was constructed by Knieper in \cite{Kni:98}. The second is the measure $\widetilde m$ described in the introduction, which is obtained by restricting the Liouville measure to $\cR$ and normalizing.

 \begin{lemma}\label{alpha1lem}
 $\alpha_1 = \chi(\widetilde m) > 0$.
 \end{lemma}

 \begin{proof}The function $\varphi^u$ is continuous and nonpositive by Lemma~\ref{nonneglemma}. Since $\widetilde m$ is absolutely continuous and has positive density throughout $\cR$, we can have $\chi(\widetilde m) = 0$ only if $\varphi^u = 0$ throughout $\cR$.
 But then the curvature of $M$ would vanish everywhere, which is impossible since $M$ has negative Euler characteristic.
 \end{proof}

 It is important for us that the forward exponent $\chi_+$ is constant on leaves of the foliation $W^{0s}$ and the backward exponent $\chi_-$ is constant on leaves of the foliation $W^{0u}$. We give the proofs for the forward exponent.

 \begin{proposition}\label{cor:exp}
	Suppose $w$ and $w'$ are in the same leaf of the foliation $W^s$. Then $w$ and $w'$ have the same forward Lyapunov exponent. This means that the
forward Lyapunov exponent is defined at $w$ if and only if  it is defined at $w'$, and the two exponents agree if they are both defined.
\end{proposition}

\begin{proof}
	The forward Lyapunov exponent is the Birkhoff average of the  {\em continuous} function~$-\varphi^u$, so the claim is obvious if $d(g^t(w),g^t(w')) \to 0$ as $t \to \infty$. Otherwise it follows from Lemma~\ref{flatstriplemma} that both $w$ and $w'$ have forward exponent $0$.
\end{proof}

\begin{corollary}\label{p:minimal}
	If $v \in \cH$, then all vectors $w\in W^{0s}(v)$ have forward exponent $\chi_+(w)=0$ and all vectors $w\in W^{0u}(v)$ have backward exponent $\chi_-(w)=0$.
	If $v\in\cL(\alpha)$, then all vectors $w\in W^{0s}(v)$ have forward exponent $\chi_+(w)=\alpha$ and all vectors $w\in W^{0u}(v)$ have backward exponent $\chi_-(w)=\alpha$.
\end{corollary}

\begin{proof}
The forward and backward exponents are obviously both $0$ for any $v \in \cH$, so the result follows from Proposition~\ref{cor:exp} and its analogue for $W^u$. The case $v\in\cL(\alpha)$ is analogous.
\end{proof}

\begin{lemma}\label{lem:exp}
Let $\gamma$ be a closed geodesic with period $\tau$ and intial tangent vector $v$. Then
 $$
 \chi(v) \leq \sqrt{- \frac1\tau \int_0^\tau K(\gamma(t))\,dt}.
 $$
\end{lemma}

\begin{proof} Consider a nonzero $\xi \in F^u_v$ and the corresponding Jacobi field $J_\xi$.
Let $j(t) = \|J_\xi(t)\|$. The bundle $F^u$ is one dimensional and invariant under the derivative of the geodesic flow. Since $\gamma$ has period $\tau$, we see using this fact  that $(J_\xi(\tau),J_\xi'(\tau))$ is a multiple of $(J_\xi(0),J_\xi'(0))$.
It follows easily that
 $$
 \chi(v) = \frac1\tau \log \left(\frac{j(\tau)}{j(0)}\right).
 $$
 The function $u(t) = j'(t)/j(t)$ is $\tau$-periodic and  by \eqref{er} satisfies the Riccati equation
 $$
 u'(t) + u(t)^2 + K(\gamma(t)) = 0.
 $$
Using the Schwarz inequality and integrating this equation gives
 $$
 \frac1\tau \log \left(\frac{j(\tau)}{j(0)}\right) = \frac1\tau\int_0^{\tau} u(t)\,dt
 \leq \sqrt{\frac1\tau\int_0^{\tau} u(t)^2\,dt} = \sqrt{-\frac1\tau\int_0^{\tau} K(\gamma(t))\,dt}.
 $$
 This proves the lemma.
\end{proof}

We use the previous lemma and the shadowing lemma from the previous subsection to investigate Lyapunov exponents.

\begin{proposition}\label{lem:exist-1}
	$\cH\ne\emptyset$ if and only if there exist closed geodesics with arbitrarily small positive Lyapunov exponent.
\end{proposition}

\begin{proof}
	Assume that $\cH\ne\emptyset$. Consider $a>0$ small enough so that there are points of $M$ at which the curvature is less than $-a^2$. Choose a vector $v\in\cH$ and a sequence $\gamma_k$ of geodesics such that $\dot\gamma_k(0) \to v$ in $T^1M$. Since the curvature of $M$ is $0$ at all points of $\gamma_v$, we see that for all large enough $k$ the times $t_k^+=\inf\{t>0\colon K(\gamma_{k}(t))<-a^2/2\}$ and
	  $t_k^-=\sup\{t<0\colon K(\gamma_{k}(t))<-a^2/2\}$
are well-defined. Moreover $t_k^+\to\infty$ and $t_k^-\to-\infty$ as $k\to\infty$. By passing to a subsequence, we may assume that the sequences $\dot\gamma_{k}(t_k^-)$ and $\dot\gamma_{k}(t_k^+)$ converge to vectors $w_-\in T_{p_-}M$ and $w_+\in T_{p_+}M$, respectively. Note that
$K(p_\pm)=-a^2/2$ and $K(\gamma_k(t)) \ge -a^2/2$ for $t_k^-< t< t_k^+$.

Given $\varepsilon>0$, choose $\delta>0$ as in the conclusion of Lemma~\ref{lem:dosha}  applied to the points $p_-$ and $p_+$ and some fixed time $\tau>0$.
	Fix a vector $w\in\cR$ that satisfies $d(w,w_+)<\delta/2$ and for which there exists a time $T>\tau$ such that $d(\dot\gamma_w(T),w_-)<\delta/2$.
 By Lemma~\ref{lem:dosha},  there is for any large enough $k$ a closed geodesic $\beta_k$ with period $\tau_k$ close to
$t^+_k - t^-_k + T$ such that $d(\dot\beta_k(t),\dot\gamma_k(t))<\varepsilon$ for $t^-_k \le t \le t^+_k$ and
$d(\dot\beta_k(t),\dot\gamma_w(t)) <\varepsilon$ for $t^+_k \leq t \le  t^+_k + T$.  Our construction ensures that
 $$
 \frac1{\tau_k} \int_0^{\tau_k} K(\beta_k(t))\,dt \to 0\qquad\text{as $k \to \infty$.}
 $$
 It follows immediately from  Lemma~\ref{lem:exp} that $\chi(\dot\beta_k(0)) \to 0$ as $k \to \infty$.

Conversely, assume that $\cH$ is empty. Then $T^1M=\cR$ and $G|_{T^1M}$ is an Anosov flow.
In particular $\inf_{v\in T^1M}\underline\chi(v) > 0$.
\end{proof}

A similar but easier  argument using shadowing proves:

\begin{proposition} \label{chioverlineprop}
	There exist closed geodesics with Lyapunov exponent arbitrarily close to $\overline\chi$ and closed geodesics with Lyapunov exponent arbitrarily close to $\underline\chi$.
\end{proposition}

These results allow us to show the extreme exponents $\underline\chi$ and $\overline\chi$ are realized by ergodic measures.

\begin{corollary}\label{expmeas}
There are ergodic measures $\underline\mu$ and $\overline\mu$ such that $\underline\chi = \chi(\underline\mu)$ and $\overline\chi = \chi(\overline\mu)$.
\end{corollary}

\begin{proof}We give the proof for $\underline\chi$; the proof in the other case is analogous. By Proposition~\ref{chioverlineprop} there are measures $\mu_n$ supported on closed orbits such
that $\chi(\mu_n) \to \underline\chi$. Let $\mu$ be a weak limit of the $\mu_n$. Then $\chi(\mu) = \underline\chi$ since $\chi$ is the Birkhoff average of the continuous function $\varphi^u$. Since $\chi(\mu)$ is an average of the Lyapunov exponents for the measures in the ergodic decomposition of $\mu$ and the exponent of any invariant measure is at least $\underline\mu$,
we can choose $\underline\mu$ to be any measure in the ergodic decomposition of $\mu$.
\end{proof}

\subsection{Entropy and pressure}\label{sec:press-1}

Let $\varphi\colon T^1M\to\bR$ be a continuous function (called the potential). For every $Z\subset T^1M$ we denote by $P(\varphi,Z)$ the \emph{topological pressure} of the potential $\varphi$ on the set $Z$ with respect to the flow $G$ and by $h(Z)$ the \emph{topological entropy} of $G$ on $Z$. The entropy is the pressure for the potential $\varphi = 0$. For simplicity we write $P(\varphi)= P(\varphi,T^1M)$ and $h=h(T^1M)$.

Bowen and Ruelle \cite{BowRue:75} gave a definition $P(\varphi,Z)$ in the special case when $Z$ is compact and invariant, and observed that their definition is equivalent to defining $P(\varphi,Z)$ as the topological pressure on $Z$ of the function $\varphi^1: v \mapsto \int_0^1\varphi(g^t(v))\,dt$ with respect to the time-$1$ map $g^1$ of the flow $G$. We use this latter definition in the general case.

This requires a definition of the  topological pressure on a set $Z$ of a potential  with respect to a continuous mapping that makes sense even if the set $Z$ is not required to be either compact or invariant. A suitable definition
 was given  by Pesin and Pitskel~\cite{PesPit:84}; \cite[Chapter 4]{Pes:97} is a convenient reference.
Their definition generalizes the definition of topological entropy on a general set for a continuous mapping given  by Bowen in~\cite{Bow:73b} and the classical notion of pressure on a compact invariant set, which can be found in~\cite{Wal:81}. In this paper we need only to consider pressure on compact invariant sets and entropy on the sets $\cL(\alpha)$, which are invariant but in general noncompact.

The properties of pressure and entropy that are familiar in the classical setting of a compact invariant set (see e.g.~\cite{Wal:81}) still apply. In particular $P(\varphi,Z_1)\le P(\varphi,Z_2)$ if $Z_1\subset Z_2$. One can also verify Abramov's formula \cite{Abr:59}
 $$
 \text{$\lvert t\rvert \,h(Z)=h(g^t,Z)$ for every $t \in \bR$,}
 $$
 in which $h(g^t,Z)$ denotes the entropy on $Z$ of the time-$t$ map $g^t$.
See \cite{Ito:69} for a proof in the case when $Z$ is compact and invariant.

As mentioned in the introduction, for a measure $\mu \in \cM$ and a potential $\varphi$, we define
 $$
 P(\varphi,\mu) \eqdef h(\mu) + \int \varphi \,d\mu,
 $$
 where $h(\mu)$ is the entropy of the time-$1$ map $g^1$ with respect to the measure $\mu$. This definition makes sense as long as $\mu$ is invariant under $g^1$, even if $\mu$ is not invariant under $g^t$ for all $t \in \bR$.

Given a set $Z\subset T^1M$, let $\cM(Z) \subset \cM$ be the set of all invariant measures $\mu$ for which $\mu(Z)=1$. We denote by $\cM_{\rm e}(Z)$ the subset of all ergodic measures in $\cM(Z)$ and by $\cM_{\rm e}^1(Z)$ the subset of measures that are ergodic for the time-$1$ map $g^1$.
Pesin and Pitskel~\cite[Theorem 1]{PesPit:84}  proved the following variational inequality for any measurable subset $Z$:
\begin{equation}\label{e.pesinpitskel}
	\sup_{\mu\in\cM(Z)}P(\varphi,\mu)
	\le P(\varphi,Z).
\end{equation}
They also gave a sufficient condition on $Z$ for this inequality to be equality for any $\varphi$. Their condition does not apply to the sets $\cL(\alpha)$, which are our primary interest, but it does apply if $Z$ is compact and invariant. In that case we have the variational principle
\begin{equation}\label{varprinc}
	P(\varphi,Z)
	= \sup_{\nu\in \cM_{\rm e}^1(Z)} P(\varphi^1,\nu)
		= \sup_{\mu\in \cM_{\rm e}(Z)} P(\varphi,\mu).
\end{equation}
The first equality is  the classical variational principle for maps; for the second see~\cite[Lemma 2]{Ohn:80}.

 A measure $\mu \in \cM_{\rm e}(Z)$ is said to be an \emph{equilibrium state} for $\varphi$ on the compact invariant set  $Z$ if $\mu$ realizes the supremum in~\eqref{varprinc}. Notice that such a measure always exists because the entropy map $\nu\mapsto h(g^1,\nu)$ is upper semi-continuous~\cite{Wal:81} since the flow is smooth~\cite{New:89}. Moreover, the  map $g^1$ is $h$-expansive~\cite[Proposition 3.3]{Kni:98} which also implies upper semi-continuity.
 In Section~\ref{s:6} we will provide further, equivalent, characterizations of the pressure and, in particular, will study the topological pressure on basic sets.

\section{Pressure spectrum for $\varphi^u$}\label{s:3}

\subsection{Pressure function in the general case}

In the following we analyze the pressure of parameterized family of
continuous potentials $q \varphi^u\colon  T^1M\to\bR$ . Let
\[
\cP(q) \eqdef P(q \varphi^u).
\]
We have $\cP(0)=h$, which is the topological entropy of the geodesic
flow on the unit tangent bundle. 
For each invariant measure $\mu \in
\cM$, we can consider the linear function $\cP_\mu: q \mapsto
P(q\varphi^u,\mu)$. By the variational principle for the topological
pressure~\eqref{varprinc},  we have $\cP \geq \cP_\mu$ for any
$\mu  \in \cM$ and we obtain the function $\cP$ by taking the
supremum of the functions $\cP_\mu$ over the ergodic measures:
 $$
 \cP = \sup \{ \cP_\mu : \mu \in \cM_{\rm e}\}.
 $$
The ergodic measures which play an active role in this supremum are called {\em equilibrium states}. More precisely $ \mu \in \cM_{\rm e}$ is an equilibrium state for $q\varphi^u$ if $\cP_\mu(q) = \cP(q) = P(q\varphi^u)$. It is also allowed for the graph of $\cP_\mu$ to be asymptotic to the graph of $\cP$; in this case we can think that $\cP_\mu(q) = \cP(q)$  when $q = \pm \infty$.

The next proposition summarizes some general consequences of the variational principle.

 \begin{proposition}\label{convprop}
	 The function $\cP$ is nonincreasing, convex. Moreover:
 	\begin{enumerate}

	\item \label{P1} $\cP$ is differentiable for
all but at most countably many $q$, and the left and right derivatives $D_L\cP(q)$ and
$D_R\cP(q)$ are defined for all $q$.

	\item \label{P2} We have
\[	\underline\chi
	=-\lim_{q\to\infty}\frac{\cP(q)}{q}
	=-\lim_{q\to\infty}D_L\cP(q)
	\,,\quad
	\overline\chi
	= -\lim_{q\to-\infty}\frac{\cP(q)}{q}
	=-\lim_{q\to-\infty}D_R\cP(q)\,.
\]

	\item  \label{P3} The graph of $\cP$ has a supporting line of slope
$-\chi(\mu)$ for every $\mu  \in \cM(\Lambda)$ Thus there is a
supporting line of slope $-\alpha$ for all $\alpha \in
[\underline\chi,\overline\chi]$.

	\item  \label{P4} If $\mu$ is an equilibrium state for $q\varphi^u$ for
some $q$, then the graph of $\cP_\mu$ is a supporting line for the graph
of $\cP$ at $(q,\cP(q))$.

	\item  \label{P5} If $\mu$ is an equilibrium state for $q\varphi^u$, then $-D_L\cP(q) \geq \chi(\mu) \geq -D_R\cP(q)$.

	\item  \label{P6} $\cP$ is differentiable at $q$ if and only if all
equilibrium states for $q\varphi^u$ have the same exponent and this
exponent is $-\cP'(q)$. In particular $\cP$ is differentiable at $q$ if
there is a unique  equilibrium state for $q\varphi^u$.

	\item  \label{P7} If $\mu$ is not ergodic and $\cP_\mu(q) = \cP(q)$ for
some $q$, then all of the measures in the ergodic decomposition of $\mu$
are equilibrium states for  $q\varphi^u$.

	\item  \label{P8} For any  $q$ there are
equilibrium states $\mu_{L,q}$ and $\mu_{R,q}$ for $q\varphi^u$
	such that  $\chi(\mu_{L,q}) = -D_L\cP(q)$ and $\chi(\mu_{R,q}) = -D_R\cP(q)$.
	
	\item \label{P9} For any $\alpha \in (\underline\chi,\overline\chi)$ there is a measure $\mu_\alpha$ 
	such that $\chi(\mu_\alpha) = \alpha$ and $q \mapsto h(\mu_\alpha) - q\chi(\alpha)$ 
	is a supporting line for $\cP$.

\end{enumerate}
\end{proposition}

 \begin{proof}The function $\cP$ is nonincreasing because $\varphi^u \leq 0$. Convexity of $\cP$ is an immediate consequence of the
variational principle and the linearity of the functions $\cP_\mu$.
 
(\ref{P1}) describes standard facts about convex functions. 
To prove (\ref{P2}), notice that $h(\mu) \ge 0$ and the variational principle imply	
	\[
	\frac 1 q\sup_{\mu\in \cM_{\rm e}} (-q\,\chi(\mu))
	\le \frac{\cP(q)}{q}
	\le \frac{h}{q} + \frac 1 q \sup_{\mu\in \cM_{\rm e}}(-q\, \chi(\mu))
	\]
	for any $q>0$ and so the first equality holds. The second one follows similarly.

(\ref{P3})
follows because the graph of $\cP_\mu$ is  either a supporting line for
the graph of $\cP$ or lies below the graph; in the latter case, there
will be a higher parallel line that is a supporting line for the graph of
$\cP$. (\ref{P4}) is immediate from the variational principle and the
definition of equilibrium states. (\ref{P5}) and (\ref{P6}) follow from
(\ref{P4}) and the fact that $\chi(\mu)$ is the derivative of the
function $\cP_\mu$.

To see (\ref{P7}), observe that if  $\mu$ is a weighted average of ergodic measures $\mu_a$ for $a \in A$, then $P(q\varphi^u,\mu)$ is a weighted average of the functions  $ P(q\varphi^u,\mu_a)$ with the same weighting  and $\cP_\mu(q)$ is a weighted average of the values $\cP_{\mu_a}(q)$, again with the same weighting. But we have $\cP_{\mu_a}(q) \leq \cP(q)$ since this is true for any invariant measure. The only way for the average to be equal to the supremum is to have
$\cP_{\mu_a}(q) = \cP(q)$
 for all $a$.

 (\ref{P8}) is immediate from (\ref{P6}) if $\cP$ is differentiable at $q$. If not, choose a sequence $q_n$ such that $q_n$ decreases to
$q$ and $\cP$ is differentiable at $q_n$ for each $n$. For each $n$ let
$\mu_n$ be the equilibrium state for $q_n\varphi^u$ and consider a weak
limit $\mu$ of the sequence $\mu_n$. Then $\cP_\mu(q) = \cP(q)$ and
$$
\int \varphi^u \,d\mu =  D_R\cP(q).
$$
All measures in the ergodic decomposition of $\mu$ are equilibrium states
for $q\varphi^u$ and have exponent at least
 $-D_R\cP(q)$. Since the exponents of
  these measures have an average equal to the lower bound $-D_R\cP(q)$, they are all equal to $-D_R\cP(q)$.
A similar argument can be made using a sequence approaching $q$
from the left.

Finally, to prove (\ref{P9}), observe using (\ref{P2}) that if $\alpha \in (\underline\chi,\overline\chi)$, then
the supporting line for $\cP$ with slope $-\alpha$ will touch the graph of $\cP$ at a point 
$(q_\alpha, \cP(q_\alpha))$. The desired measure $\mu_\alpha$ is the appropriate linear combination of the measures $\mu_{L,q_\alpha}$ and $\mu_{R,q_\alpha}$ provided by (\ref{P8}).

\end{proof}

\begin{lemma}\label{cornerlemma}
We have $\cP(1) = 0$.
If $\cH \neq \emptyset$, then $\cP(q) = 0$ for all $q \geq 1$,
and the supporting line for $\cP$ of slope $-\alpha$ passes through $(1,0)$ for 
$0 \leq \alpha \leq \alpha_1 \eqdef \chi(\widetilde m)$.
\end{lemma}

\begin{proof} The measure $\widetilde m$, which is the restriction of the Liouville
measure to $\cR$ normalized to be a probability measure, is absolutely
continuous with respect to the Liouville measure, and thus
 $h(\widetilde m) = \chi(\widetilde m)$ by Pesin's formula. On the other
hand, the Ruelle inequality gives $h(\mu) \leq \chi(\mu)$ for any
invariant probability measure. It follows that $\widetilde m$ is an
equilibrium state for $\varphi^u$ and thus $\cP(1) = 0$.

 The set $\cH$ is closed and invariant. If it is nonempty, $\cH$ supports
at least one invariant measure. We have $\chi(\mu) = 0$ for any such
measure $\mu$, since the curvature is $0$ at the footpoint of any vector
in $\cH$. It follows from this, the Ruelle inequality and the variational
principle for entropy that  we also have $h(\mu) = 0$ for any measure
$\mu$ supported on $\cH$, and hence
 $\cP_\mu(q) = 0$ for all $q$. Thus if $\cH \neq \emptyset$, there is a
measure $\mu$ such that $\cP_\mu = 0$. It follows from the variational
principle that $\cP(q) \geq 0$ for all $q$ in this case.
 Furthermore, since $\cP(1) = 0$ and $\cP$ is nonincreasing, we must have
$\cP(q) = 0$ for all $q \geq 1$.

We see that $D_L\cP(1) \leq - \chi(\widetilde m)$ and $D_R\cP(1) = 0$ if $\cH \neq \emptyset$.
It follows immediately that if $0 \leq \alpha \leq \chi(\widetilde m)$, then the supporting line for $\cP$ of slope 
$-\alpha$ must pass through $(1,0)$.
\end{proof}

 Since $\alpha_1 > 0$ by 
Lemma~\ref{alpha1lem},  the graph of $\cP$ has a corner at $(1,0)$ if $\cH \neq \emptyset$. On the other hand,  Knieper showed that there is a unique measure of maximal entropy in our situation \cite{Kni:98}; it follows from this and part (\ref{P6}) of Proposition~\ref{convprop} that $\cP$ is differentiable at $0$.

We do not know if it is possible for there to be other corners in the graph when $\cH \neq \emptyset$. We also do not know if $\cP$ must be strictly convex on $(-\infty,1)$ or if $D_L\cP(1)=-\chi(\widetilde m)$. 

\begin{figure}
\begin{minipage}[c]{\linewidth}
\centering
a)\begin{overpic}[scale=.50
  ]{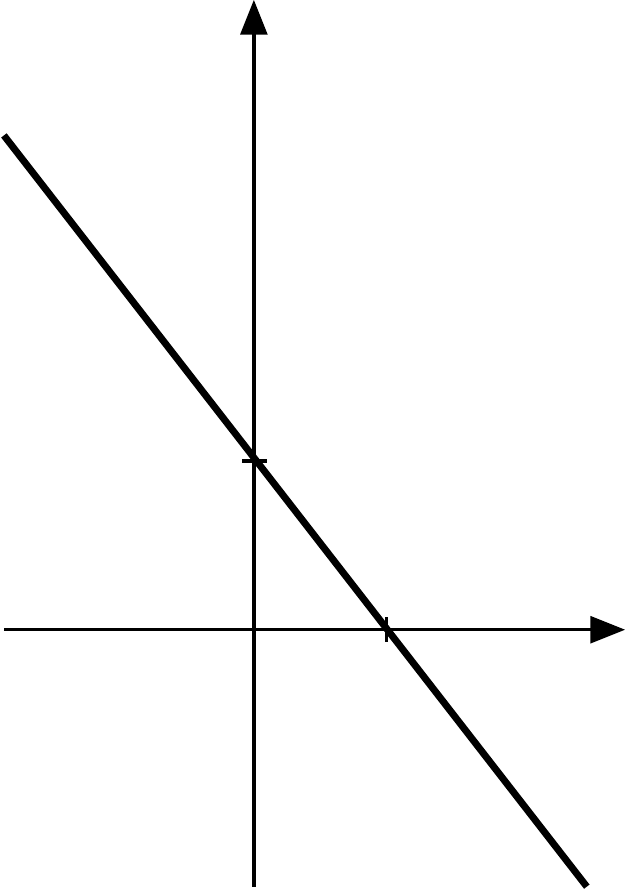}
      	\put(42,20){\tiny$1$}	
      	\put(72,27){\small$q$}	
      	\put(35,90){\small$\cP(q)$}	
	\put(-2,46){\tiny$h=k$}
\end{overpic}
\hspace{0.3cm}
b)\begin{overpic}[scale=.50
  ]{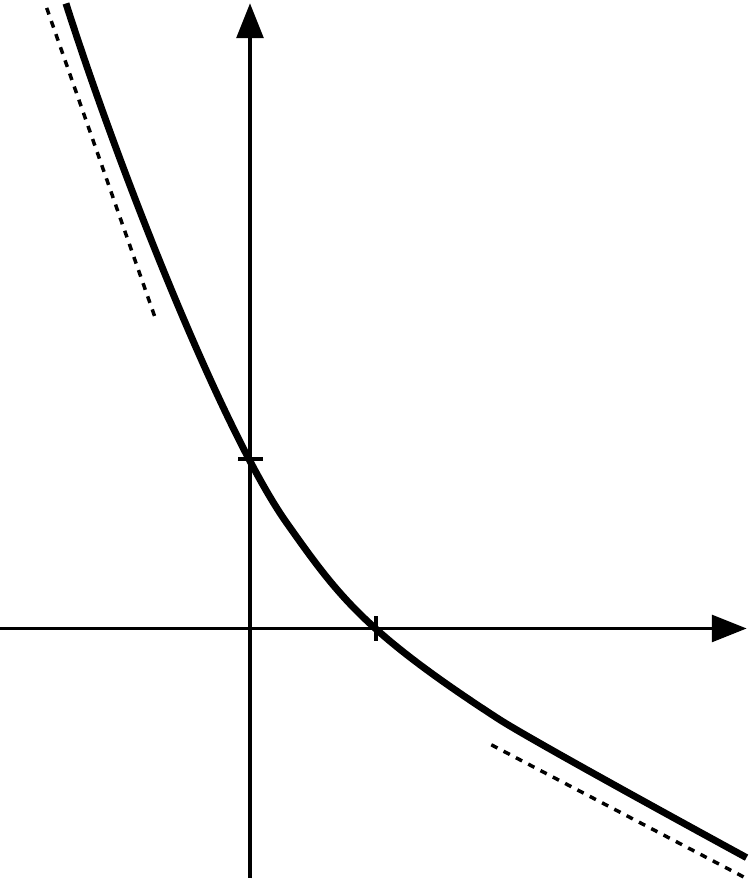}
      	\put(41.5,20){\tiny$1$}	
      	\put(87,27){\small$q$}	
      	\put(35,90){\small$\cP(q)$}	
	\put(10,46){\tiny$h$}
\end{overpic}
\hspace{0.3cm}
c)\begin{overpic}[scale=.50
  ]{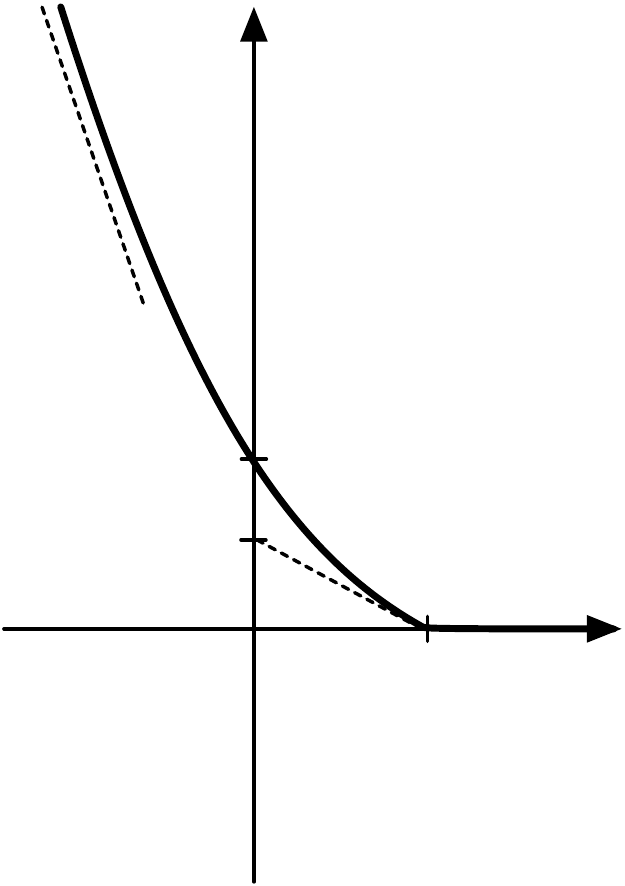}
  	\put(47,20){\tiny$1$}	
      	\put(73,27){\small$q$}		
      	\put(35,90){\small$\cP(q)$}
	\put(10,46){\tiny$h$}
\end{overpic}
\vspace{0.5cm}
\end{minipage}
\caption{The pressure function $q\mapsto \cP(q)=P(q\varphi^u)$ in the case of a)  constant negative curvature $K=-k^2$, b) negative curvature $-k_1^2\le
K\le -k_2^2$, c) non-positive curvature $-k^2\le K\le 0$}
\label{fig.map33}
\end{figure}
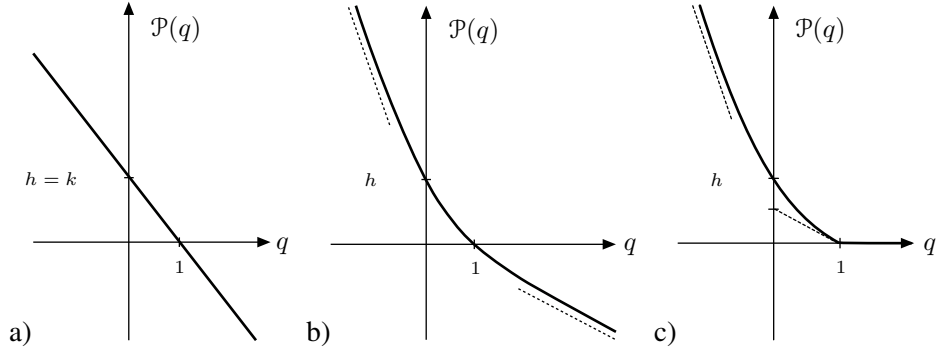

 Figure~\ref{fig.map33} shows the function $\cP$ in three cases. When the
curvature is everywhere $-k^2$, all invariant measures have exponent $k$
and $\cP$ is linear. When the curvature is negative and variable, $\cH =
\emptyset$ and the whole of $T^1M$ is a basic set. The function $\cP$ is
strictly convex and real analytic~\cite{BowRue:75}. If $-k_1^2\le K\le -k_2^2$, the
exponent of any invariant measure lies in the interval $[k_1,k_2]$ and
$\cP$ is decreasing.
 When $\cH \neq \emptyset$,   the graph has the corner at $(1,0)$ and
$\cP(q)$ is constant for $q \geq 1$.

\subsection{Pressure function for basic sets}\label{sec:3.1basicsets}

If $\Lambda$ is a basic set for $G$, we can make an analogous definition
of the functions
 $$
	 \cP_\Lambda(q) \eqdef P(q\varphi^u,\Lambda)
 	\quad\text{ and }\quad
	\cE_\Lambda(\alpha)\eqdef \inf_{q\in\bR}(\cP_\Lambda(q)+q\alpha).
 $$
 All of the properties listed in Proposition~\ref{convprop} apply (with suitable minor modifications) to $\cP_\Lambda$.
The function $\cP_\Lambda$ also enjoys additional properties that we collect below.
These results go back to Bowen and
Ruelle \cite{BowRue:75}. They are proved by studying the symbolic
dynamics provided by a Markov partition of $\Lambda$; the equilibrium
measures are Gibbs states. The rich supply of equilibrium measures on a
basic set will be crucial in our arguments.

\begin{proposition}\label{convpropforbasic}
	The function $\cP_\Lambda$ is strictly convex. Moreover:
 	\begin{enumerate}

	\item\label{PL1} $\cP_\Lambda$ is real analytic on $\bR$.

	\item\label{PL2} There is a unique supporting line of slope $-\alpha$ for all $\alpha \in [\underline\chi(\Lambda),\overline\chi(\Lambda)]$.

	\item\label{PL3} For every $q$ there is a unique equilibrium state $\mu_q$ for $q\varphi^u$ 
and it satisfies
	\[
		\cE_\Lambda(\chi(\mu_q))   = h(\mu_q)
	\]

	\item\label{PL4} $\chi(\mu_q)= -{\cP_\Lambda}'(q)$.




	\item[5.] \label{PL8} The domain of $\cE_\Lambda$ coincides with the range of $-{\cP_\Lambda}'$. We have
\[	\underline\chi(\Lambda)
	=-\lim_{q\to\infty}\frac{\cP_\Lambda(q)}{q}
	\,,\quad
	\overline\chi(\Lambda)
	= -\lim_{q\to-\infty}\frac{\cP_\Lambda(q)}{q}\,.
\]
\end{enumerate}
\end{proposition}

The function $\cP_\Lambda$ contains the entire information about the Hausdorff dimension and the topological entropy of every level sets:   for every $\alpha \in (\underline\chi(\Lambda),\overline\chi(\Lambda))$ we have
\begin{equation}\label{dimfors}
\begin{split}
		\dim_{\rm H}\big( \cL(\alpha) &\cap\Lambda\big)
			= 1+2\,\cD_\Lambda(\alpha)\\
		h(\cL(\alpha)\cap\Lambda)
		= \cE_\Lambda(\alpha)
		=& \max\Big\{ h(\mu)\colon
			\chi(\mu)=\alpha\text{ and }\mu\in\cM(\Lambda)\Big\}
\end{split}
\end{equation}
(compare~\cite{BowRue:75},~\cite{PesSad:01},~\cite[Section 6]{BarDou:04}).
Moreover, the graph of
the function $\cP_\Lambda$ crosses the $q$-axis at $q = \delta_\Lambda$,
where $\delta_\Lambda$ is the Hausdorff dimension of the intersection of
$\Lambda$ with the unstable (or stable) manifold of a typical point in
$\Lambda$. The Hausdorff dimension of $\Lambda$ itself is $1 +
2\delta_\Lambda$.

\subsection{The convex conjugate $\cE$}

In this section we study the function
\[
	\cE(\alpha) =
	\inf_{q \in \bR}  \left( P(q\varphi^u) + q\alpha\right) .
\]
As noted in the introduction, $\cE$ is the conjugate of the convex function $\cP$ under the Legendre-Fenchel transform. In particular $\cE$ is concave.

\begin{lemma}\label{Echarac}
	For any $\alpha\in(\underline\chi,\overline\chi)$  we have
	\[
	\cE(\alpha)=
	\max\big\{h(\mu)\colon \mu\in\cM\text{ and } \chi(\mu)=\alpha\big\}.
	\]
\end{lemma}

\begin{proof} By (\ref{P9}) of Proposition~\ref{convprop} there is a measure $\mu_\alpha$ such that $\chi(\mu_\alpha) = \alpha$ and 
 $$
 q \mapsto h(\mu_\alpha) - q\chi(\mu_\alpha)
 $$
 is a supporting line for the convex function $\cP$. Since $\cE$ is the conjugate of $\cP$, we obtain $\cE(\alpha) = h(\mu_\alpha)$.
 
 On the other hand, by the variational principle, the line $q \mapsto h(\mu) - q\chi(\mu)$ lies below the graph of $\cP$ for any invariant measure $\mu$. It follows that $\cE(\alpha) \geq h(\mu)$ for any measure $\mu$ with
 $\chi(\mu) = \alpha$.
	\end{proof}
	
	Now we can prove Theorem~\ref{entupperboundthm} that $h(\cL(\alpha)) \leq  \cE(\alpha)$ for $\alpha \in [\underline\chi,\overline\chi]$.

\begin{proof}[Proof of Theorem~\ref{entupperboundthm}]
For $v \in T^1M$, let $\cM(v)$ be the set of measures $\mu$ that are weak limits as $T \to \infty$ of the measures $\mu_{v,T}$ defined by
\[
	\int \varphi\,d\mu_{v,T}= \frac 1 T\int_0^T\varphi(g^t(v))\,dt
		\quad\text{ for all }\varphi\in C^0(T^1M).
\]
The definition of $\cL(\alpha)$ ensures that if $\mu \in \cM(v)$ for $v \in \cL(\alpha)$, then
\[
   \chi(\mu) = - \int \varphi^u\,d\mu = \alpha.
\]
It follows from Lemma~\ref{Echarac} that $h(\mu) \leq \cE(\alpha)$ for such a measure. We therefore have the inclusion
\[
	\cL(\alpha)
	\subset QR(\cE(\alpha))\eqdef
	\{w\in T^1M\colon \exists\,\mu\in \cM(w)\text{ with }h(\mu)\le \cE(\alpha)\},
\]
and hence $h(\cL(\alpha))\le h\big(QR(\cE(\alpha))\big)$. The sets $QR(\cdot)$ were studied by Bowen in~\cite{Bow:73b}. Theorem~2 of~\cite{Bow:73b} gives $h\big(QR(\cE(\alpha))\big)\le \cE(\alpha)$. Combining the last two inequalities proves the theorem.
\end{proof}

Finally, we can complete the proof of Theorem~\ref{Hnonzerothm}.

\begin{proof}[Proof of Theorem~\ref{Hnonzerothm}]
If $\cH \neq \emptyset$, Lemma~\ref{cornerlemma} tells us that $\cP(q) = 0$ for $q \geq 1$ and the supporting line of slope $-\alpha$ passes through $(1,0)$ for $0 \leq \alpha \leq \alpha_1$. It is immediate from the latter property and the definitions of the functions $\cD$ and $\cE$ that $\cD(\alpha) = 1$  and $\cE(\alpha) = \alpha$ for $0 \leq \alpha \leq \alpha_1$. Now we apply Theorems~\ref{mainthm} and~\ref{entupperboundthm} to obtain 
$\dim_{\rm H} \cL(\alpha) = 3$ and $h(\cL(\alpha)) = \alpha$ for $0 \leq \alpha \leq \alpha_1$.
\end{proof}

\section{Uniformly hyperbolic (sub)systems}\label{s:4}

	In this section we are going to study compact ${G}$-invariant uniformly hyperbolic sets (see~\cite{KatHas:95} for definition). In our setting, given any compact ${G}$-invariant hyperbolic set $\Lambda\subset\cR$, then over $\Lambda$ the splitting $F^s\oplus V\oplus F^u$ is the hyperbolic splitting.

\begin{lemma}\label{lem:hyp}
	A compact ${G}$-invariant set $\Lambda\subset T^1M$  is hyperbolic for ${G}$ if and only if $\Lambda \subset \cR$.
\end{lemma}

\begin{proof}
If $\Lambda \cap \cH \neq \emptyset$, then $\Lambda$ is not hyperbolic because it contains the tangent vector to a geodesic along which the curvature is always zero and hence has zero exponent.

We now suppose that $\Lambda \subset \cR$. Then there is $T > 0$ such the every geodesic segment of length $T$ tangent to a vector in $\Lambda$ contains a point at which the curvature is negative. For otherwise there would be geodesic segments tangent to vectors in $\Lambda$ of arbitrary length along which the curvature vanishes; these segments would accumulate on an  entire geodesic along which the curvature vanishes, and the tangent vector to this geodesic would be in $\cH \cap \Lambda$.

Our choice of $T$ ensures that the time $T$ map of the geodesic flow expands $F^u_v$ and contracts $F^s_v$ for all $v \in \Lambda$.
Since the bundles $F^u$ and $F^s$ are continuous it follows from the  compactness of $\Lambda$ that this expansion and contraction is uniform for all $v \in \Lambda$.
\end{proof}

\begin{proposition}\label{p:topone}
	If $\cH\ne\emptyset$ then any compact $G$-invariant set $\Lambda\subset\cR$, $\Lambda\ne\cR$, has topological dimension $1$.
\end{proposition}

\begin{proof}
	This claim can be proved using the minimality of the horocycle flow and properties of the Lyapunov exponent. Eberlein showed in~\cite{Ebe:73trans} that the foliations $W^s$ and $W^u$ are minimal.
	It follows immediately that if $v \in T^1M$ is a vector such that the spaces $F^s_v$, $F^u_v$, and the generator $V(v)$ of the geodesic flow span $T_vT^1M$ and if $w$ is any vector in $T^1M$, then $W^{0s}(w) \cap W^{u}(v)$ is dense in a neighborhood of $v$ in $W^{u}(v)$, and $W^{0u}(w) \cap W^{s}(v)$ is dense in a neighborhood of $v$ in $W^{s}(v)$. In particular these properties hold for any $v \in \cR$.

Now suppose that $\Lambda \subset \cR$ is a compact invariant set and $\Lambda\ne\cR$.
Choose $v \in \Lambda$ and $w \notin \Lambda$. Then $v$ lies in the interior of arbitrarily short arcs in $W^s(v)$ whose endpoints $w_s$ and $w'_s$ belong to $W^{0u}(w)$. Also $v$ lies in the interior of arbitrarily short arcs in $W^u(v)$ whose endpoints $w_u$ and $w'_u$ belong to $W^{0s}(w)$.
Let $S$ be a smooth disc in $T^1M$ transverse to the geodesic flow at $v$. Then $v$ lies in the interior of an open set in $S$ bounded by four arcs that lie in the intersections of $S$ with $W^{0u}(w_s)$, $W^{0u}(w'_s)$, $W^{0s}(w_u)$, and $W^{0s}(w'_u)$.
In particular, when choosing $w\in \cH$, it follows from Corollary~\ref{p:minimal} that all of the vectors in these arcs have a forward or backward exponent that is $0$. The arcs are therefore disjoint from the hyperbolic set $\Lambda$. This shows that the intersection of $\Lambda$ with the transversal $S$ has topological dimension $0$.
It follows immediately that the flow invariant set $\Lambda$ has topological dimension $1$.
This finishes the proof of the proposition.
\end{proof}

Using the classical approach (see for example~\cite{PesSad:01}, in order to calculate multifractal properties in some compact $G$-invariant set $\Lambda\subset T^1M$ it is crucial that this set is also hyperbolic, topologically transitive, and locally maximal. For that recall that ${G}|_\Lambda$ is \emph{topologically transitive} if for any nonempty open sets $U$ and $V$ intersecting $\Lambda$ there exists $t\in\bR$ such that ${g}^t(U)\cap V\cap \Lambda\ne\emptyset$.
Further recall that $\Lambda\subset  T^1M$ is said to be \emph{locally maximal} if there exists a neighborhood $U$ of $\Lambda$ such that $\Lambda=\bigcap_{t\in\bR}{g}^t(\overline U)$.

The following proposition is one of the key results in our approach. Note that sets in its hypothesis indeed exist by the construction~\eqref{guy} below and by Proposition~\ref{p:topone}.

\begin{proposition}\label{prop:exis2}
	Let $\Lambda \subset \cR$ be a compact $G$-invariant set of topological dimension $1$. Then every neighborhood of $\Lambda$ contains a compact $G$-invariant hyperbolic locally maximal set $\widetilde\Lambda$ of topological dimension $1$ such that  $\Lambda \subset \widetilde\Lambda$.
\end{proposition}

\begin{proof}
We use an argument of Anosov~\cite{Ano:}.
The set  $\Lambda$ is hyperbolic and we can assume that the neighborhood of $\Lambda$ is small enough so that any invariant set contained in it is also hyperbolic. It is well known that a compact hyperbolic set is locally maximal if and only if it has a local product structure. Rather than constructing $\widetilde\Lambda$ itself, we construct its intersection with a suitable cross section to the flow $S$. This intersection is obtained as the image of a subshift of finite type $\Sigma$  under a continuous and injective map $\psi$. This ensures that $\widetilde\Lambda \cap S$ is compact, since it is the continuous image of a compact space. Moreover
the subshift $\Sigma$ has a natural local product structure: the product of two bi-infinite sequences with the same zeroth term is formed by concatenating the future half of one sequence and the past of the other sequence. The continuous injection $\psi$ will carry this local product structure to a local product structure on $\widetilde\Lambda \cap S$, which is then inherited by $\widetilde\Lambda$.
Since this subshift $\Sigma$ has topological dimension $0$, it is obvious that $\widetilde\Lambda$ will have topological dimension $1$.

We begin by choosing the cross section $S$. It is the union of a finite collection of smoothly embedded closed two dimensional discs such that
$\Lambda$ does not intersect the boundary of any disc. The discs should be
pairwise disjoint, transverse to the geodesic flow, and close enough together so that $S$ intersects every orbit segment of length $1$.
 Let $\cT: S \to S$ be the first return map of the geodesic flow to $S$ and define $\tau(v)$ for $v \in S$ by the equation $$\cT(v) = g^{\tau(v)}(v).$$
 There is $\tau_{\rm min} > 0$ such that $\tau(v) \in [\tau_{\rm min},1]$ for all $v \in S$.

 Let us call a sequence $\{X_i\}_{i = -\infty}^\infty$ of subsets of $S$ an {\em $\varepsilon$-pseudo orbit} if  for each $i$
 \begin{itemize}
 	\item $X_i$ lies in a single disc belonging to $S$;
 	\item $\diam(X_i) < \varepsilon$;
 	\item there is $v_i \in X_i \cap \Lambda$ such that $\cT(v_i) \in X_{i+1}$.
 \end{itemize}
We will say that $v \in S$ is a $\delta$-shadow of the pseudo orbit $\{X_i\}$ if for each $i$ the set $X_i$ lies in the same disc belonging to $S$ as $\cT^i(v)$ and
 $$
 d(x_i, \cT^i(v)) < \delta \qquad\text{for all $x_i \in X_i$.}
 $$
 Note that $v$ is not required to be in $\Lambda$, but the $\cT$-orbit of $v$ lies in the $\delta$-neighborhood of $\Lambda$. Since $G|_\Lambda$ is hyperbolic,
there is $\delta_0$ such that any pseudo orbit has at most one $\delta_0$-shadow, and for any positive $\delta \leq \delta_0$ there is $\varepsilon(\delta) > 0$ such that every $\varepsilon(\delta)$-pseudo orbit has a  $\delta$-shadow, which is unique. Let $\varepsilon_0 = \varepsilon(\delta_0)$.

Next choose a  collection  $\cK_0$ of subsets of $S$ such that the following is satisfied:
\begin{enumerate}
\item the collection $\cK_0$ is finite;
\item the sets in $\cK_0$ are compact and pairwise disjoint;
\item each $K \in \cK_0$ lies in one of discs in $S$ and has diameter $\le \min\{\delta_0,\varepsilon_0\}$;
\item $\cT$ is smooth on each $K\in \cK_0$;
\item each $K \in \cK_0$ contains a point of $\Lambda \cap S$;
\item $S \cap \Lambda \subset \bigcup_{K \in \cK_0} \Int K$.
\end{enumerate}
This is possible because the topological dimension of $\Lambda$ is $1$ and of $\Lambda \cap S$ is $0$.

We now define $\cK_N$ to be the collection of all sets that contain an element of $\Lambda \cap S$ and have the form
 $$
 \bigcap_{j= -N}^N \cT^{-j} K_j,
 $$
 where $K_{-N}$, $\dots$, $K_N$ is a sequence of (not necessarily distinct) sets in $\cK_0$. The analogues of properties 1--6 above hold for each $\cK_N$. Note that each element of $\cK_N$ is contained in a set from $\cK_{N-1}$.

  A bi-infinite sequence $\{K^N_i\}_{i = -\infty}^\infty$ in $ \cK_N$ will be called {\em admissible} if for each index $i$ there is a point $v_i \in K^N_i \cap \Lambda$ such that $\cT(v_i) \in K^N_{i+1}$. An admissible sequence is an $\varepsilon_0$-pseudo orbit and therefore has a unique $\delta_0$-shadow.
 Let us denote by $\psi$ the map that takes an admissible sequence to its $\delta_0$-shadow.

 We will show that $\psi$ must be injective if $N$ is large enough. We begin by showing the following lemma.

\begin{lemma} \label{sizelemma}
	The maximum diameter of a set in $\cK_N$ tends to $0$ as $N \to \infty$.
\end{lemma}

\begin{proof}
	If not then there is a sequence of sets $K^0 \supset K^1 \supset K^2 \supset \cdots$ whose diameters do not tend to $0$ with $K^N \in \cK_N$ for each $N$. Then
\[
 	K^\infty = \bigcap_{N=0}^\infty K^N
\]
 contains (at least) two distinct elements $v$ and $v'$. We may assume that $v \in \Lambda$ because
$\Lambda$ is compact and each $K^N$ contains contains an element  of $\Lambda$.  Let $K_i(v)$ be the element of $\cK_0$ that contains $\cT^i(v)$. Then
 $$
 K^\infty = \bigcap_{i = -\infty}^\infty \cT^{-i}K_i(v).
 $$
We see that $\{K_i(v)\}_{i = -\infty}^\infty$ is an admissible sequence in $\cK_0$ with two distinct $\delta_0$-shadows, namely $v$ and $v'$. This is impossible, which proves the lemma.
 \end{proof}

Now choose $\delta_1 \in (0, \delta_0)$ such that $\Lambda \cap K$ lies in the $\delta_1$-interior of $K$ for each $K \in \cK_0$ and let $\varepsilon_1 = \varepsilon(\delta_1)$.  By Lemma~\ref{sizelemma}, we can choose $N$ large enough so that every admissible sequence in $\cK_N$ is an $\varepsilon_1$-pseudo orbit and therefore has a unique $\delta_1$-shadow, which must be the image of the sequence under the map $\psi$.

Finally suppose that $v \in S$ is  the image under the map $\psi$ of  two admissible sequences $\{K^N_i\}_{i= -\infty}^{\infty}$ and $\{L^N_i\}_{i= -\infty}^{\infty}$ in $\cK_N$.
For each $i$, choose $v_i \in \Lambda \cap K^N_i$ and
let $K_i$ be the set in $\cK_0$ that contains $K^N_i$. Since $d(v_i,\cT^i(v)) < \delta_1$ and $v_i$ lies in the $\delta_1$-interior of $K_i$, we see that $\cT^i(v) \in K_i$. Similarly, $\cT^i(v) \in L_i$
for each $i$, where $L_i$ is the set in $\cK_0$ that contains $L^N_i$. Since the sets in $\cK_0$ are pairwise disjoint, we have $K_i = L_i$ for all $i$. But this gives us $K^N_i = L^N_i$ for all $i$, because
 $$
 K^N_i = \bigcap_{j = i - N}^{i+N} K_j\quad\text{and}\quad L^N_i = \bigcap_{j = i - N}^{i+N} L_j.
 $$
 This verifies that $\psi$ is injective.

 The subshift $\Sigma$ mentioned in the beginning of this proof of the proposition consists of the admissible sequences in $\cK_N$ for a suitably large $N$. The
 basic set $\widetilde\Lambda$ contains all orbits that pass through a vector in $\psi(\Sigma)$. This finishes the proof of Proposition~\ref{prop:exis2}.
\end{proof}

\section{Bridging between uniformly hyperbolic subsystems}\label{s:5}

The essential approach to prove our main result, Theorem~\ref{mainthm},  is to construct a family of basic sets $\Lambda_\ell$ that ``exhaust'' the non-uniformly hyperbolic set $T^1M$  so that for each of these sets we can determined their multifractal properties. Based on this fact we will later show convergence of corresponding quantifiers such as pressure, dimension, and spectrum of Lyapunov exponents. In this section we explain how to produce hyperbolic sets and, in particular, how to produce sufficiently large basic sets with nice properties.

First, from Lemma~\ref{lem:dosha} in particular we obtain that any closed pseudo orbit in a compact set of rank $1$ vectors is shadowed by a closed orbit of ${G}$. We will use this fact in the proof of the following proposition that shows the existence of simplest hyperbolic subsets (closed orbits) with arbitrarily small degree of hyperbolicity.

\begin{proposition}\label{lem:exist-2}
	$\cH\ne\emptyset$ if and only if there exist vectors that are tangent to closed geodesics and have arbitrarily small positive Lyapunov exponent.
	If $\cH=\emptyset$ then there exist vectors that are tangent to closed geodesics and have a positive Lyapunov exponent arbitrarily close to $\underline\chi$.
\end{proposition}

\begin{proof}
	Assume that $\cH\ne\emptyset$. Consider $a>0$ small enough so that there are points of $M$ at which the curvature is less than $-a^2$. Choose a vector $v\in\cH$ and a sequence $\gamma_k$ of geodesics such that $\dot\gamma_k(0) \to v$ in $T^1M$. Since the curvature of $M$ is $0$ at all points of $\gamma_v$, we see that for all large enough $k$ the times $t_k^+=\inf\{t>0\colon K(\gamma_{k}(t))<-a^2/2\}$ and
	$t_k^-=\sup\{t<0\colon K(\gamma_{k}(t))<-a^2/2\}$ are well-defined. Moreover $t_k^+\to\infty$ and $t_k^-\to-\infty$ as $k\to\infty$. By passing to a subsequence, we may assume that the sequences $\dot\gamma_{k}(t_k^-)$ and $\dot\gamma_{k}(t_k^+)$ converge to vectors $w_-\in T_{p_-}M$ and $w_+\in T_{p_+}M$, respectively. Note that
$K(p_\pm)=-a^2/2$ and $K(\gamma_k(t)) \ge -a^2/2$ for $t_k^-< t< t_k^+$.

Given $\varepsilon>0$, choose $\delta>0$ as in the conclusion of Lemma~\ref{lem:dosha}  applied to the points $p_-$ and $p_+$ and some fixed time $\tau>0$.
	Fix a vector $w\in\cR$ that satisfies $d(w,w_+)<\delta/2$ and for which there exists a time $T>\tau$ such that $d(\dot\gamma_w(T),w_-)<\delta/2$.
 By Lemma~\ref{lem:dosha},  there is for any large enough $k$ a closed geodesic $\beta_k$ with period $\tau_k$ close to
$t^+_k - t^-_k + T$ such that $d(\dot\beta_k(t),\dot\gamma_k(t))<\varepsilon$ for $t^-_k \le t \le t^+_k$ and
$d(\dot\beta_k(t),\dot\gamma_w(t)) <\varepsilon$ for $t^+_k \leq t \le  t^+_k + T$.  Our construction ensures that
 $$
 \frac1{\tau_k} \int_0^{\tau_k} K(\beta_k(t))\,dt \to 0\qquad\text{as $k \to \infty$.}
 $$
 It follows immediately from  Lemma~\ref{lem:exp} that $\chi(\dot\beta_k(0)) \to 0$ as $k \to \infty$.

Conversely, assume that $\cH$ is empty. Then $T^1M=\cR$ and $G|_{T^1M}$ is an Anosov flow. In particular $\inf_{v\in T^1M}\underline\chi(v) > 0$.
The claim then follows from a shadowing argument.
\end{proof}

Following similar arguments as in the proof above, the following lemma can be shown.

\begin{lemma}\label{l:largeperiodic}
	There exist vectors that are tangent to closed geodesics and have a positive Lyapunov exponent arbitrarily close to $\overline\chi$.
\end{lemma}

A compact ${G}$-invariant hyperbolic topologically transitive and locally maximal set is also called a \emph{basic set}. By the flow version of the Smale spectral  decomposition theorem, given any compact invariant hyperbolic locally maximal set $\Lambda$, there is a decomposition of the set $\Omega({G}|_\Lambda)$ of nonwandering points for ${G}|_\Lambda$ into finitely many disjoint sets $\Omega({G}|_\Lambda)=\Lambda_1\cup\ldots\cup \Lambda_k$ such that  ${G}|_{\Lambda_i}$ is topologically transitive (see, for example,~\cite[Chapter 18]{KatHas:95}). Any such component $\Lambda_i$ is closed and invariant, and hence basic.

The following proposition enables us to bridge between basic sets and to include them into a basic one.

\begin{proposition}[Bridging]\label{p:p5.3}
	Suppose $\Lambda'$ and $\Lambda''$ are basic sets of topological dimension 1.  Then there is a basic set of topological dimension 1 that contains $\Lambda' \cup \Lambda''$.
\end{proposition}

\begin{proof}
We consider the non-trivial case that the flow is not Anosov and that hence $\Lambda'$, $\Lambda''\ne\cR$. By Proposition~\ref{p:topone}, $\Lambda'$ and $\Lambda''$ have topological dimension 1.

Choose $v' \in \Lambda'$, $v''\in\Lambda''$ so that the orbits of $v'$, $v''$ are both forward and backward dense in $\Lambda'$, $\Lambda''$, respectively. Also choose $v \in \cR$ so that the orbit of $v$ is both forward and backward dense in $T^1M$. By moving $v'$ and $v''$ if necessary, we may assume that the curvature of $M$ is negative at the footpoints of these vectors. It follows from the shadowing lemma (Lemma 3.4) that we find  $w'\in \cR$ such that the orbit of $w'$ is backward asymptotic to the orbit of $v''$, forward asymptotic to the orbit of $v'$, and in between is close to a segment of the orbit of $v$. Also we can find $w'' \in \cR$ such that the orbit of $w''$ is backward asymptotic to the orbit of $v'$, forward asymptotic to the orbit of $v''$, and in between is close to a segment of the orbit of $v$.

Let $\Lambda$ be the union of $\Lambda'$, $\Lambda''$, and the orbits of $w'$ and $w''$. Then $\Lambda$ is connected compact invariant and contained in $\cR$ and has topological dimension $1$. By Proposition~\ref{prop:exis2} there is a compact invariant hyperbolic and locally maximal set $\widetilde\Lambda$ of topological dimension 1 containing $\Lambda$.

Observe that $\Lambda\subset \Omega(G|_{\widetilde\Lambda})$. Indeed, on the one hand $\Lambda'=\Omega(G|_{\Lambda'})\subset \Omega(G|_{\widetilde\Lambda})$ and analogously $\Lambda''\subset  \Omega(G|_{\widetilde\Lambda})$.
On the other hand, given $\varepsilon>0$ sufficiently small let $\delta>0$ be as in Lemma~\ref{lem:dosha}. By construction, there exist numbers $t',t'',\tau',\tau''>0$ such that $d(g^{t'}(w'),v')$,  $d(g^{-t''}(w''),v')$, $d(g^{\tau'}(w''),v'')$,  $d(g^{-\tau''}(w'),v'')<\delta/2$ and hence there exists a closed orbit that $\varepsilon$-shadows the orbit pieces $w'\mapsto g^{t'}(w')$, $g^{-t''}(w'')\mapsto w''$, $w''\mapsto g^{\tau'}(w'')$  and $g^{-\tau''}(w')\mapsto w'$. Since $\widetilde\Lambda$ is locally maximal, this orbit must lie in $\widetilde\Lambda$. This implies that $w'$ and $w''$ are non-wandering with respect to $G|_{\widetilde\Lambda}$.
By Smale's spectral decomposition $\Omega(G|_{\widetilde\Lambda})$ splits into finitely many disjoint basic sets on each of which the flow is topologically transitive. Hence the connected component of this decomposition which contains $\Lambda$ is our desired basic set.
\end{proof}

We now describe a natural way to obtain compact $G$-invariant hyperbolic subsets and, in particular, basic sets that are sufficiently large. Given a set $\Lambda$, denote by $\Per(\Lambda)$ the periodic orbits in $\Lambda$. Given $\ell\ge 1$, let
\begin{equation}\label{guy}
	\ccO_\ell \eqdef  \bigcup_{v\in\cH}B\left(v,\ell^{-1}\right)
	\quad\text{ and }\quad
	\widehat\Lambda_\ell \eqdef
	\overline{\Per \big(T^1M\setminus \ccO_\ell\big)}.
\end{equation}
By Lemma~\ref{lem:hyp} any $\widehat\Lambda_\ell$ is compact ${G}$-invariant  hyperbolic, and  $\widehat\Lambda_\ell\subset \widehat\Lambda_{\ell+1}$ for all $\ell$.

\begin{lemma}\label{lem:dense}
	 $\bigcup_{\ell\ge 1}\widehat\Lambda_\ell$ is dense in $T^1M$.
\end{lemma}

\begin{proof}
	The vectors tangent to rank 1  closed geodesics are dense in $T^1M$~\cite{Bal:82}. Any  such  vector is in $\widehat\Lambda_\ell$ for some $\ell\ge 1$.
\end{proof}

Notice that $\widehat\Lambda_\ell$ is not necessarily basic, but it is always contained in a basic set.

\begin{corollary}\label{cor:allincluded}
	For every $\ell\ge1$ there exists a basic set of topological dimension 1 containing $\widehat\Lambda_\ell$.
\end{corollary}

\begin{proof}
	By construction, $\widehat\Lambda_\ell\subset\cR$, $\widehat\Lambda_\ell\ne\cR$. By Propositions~\ref{p:topone} and~\ref{prop:exis2}, hence there exists a compact $G$-invariant hyperbolic locally maximal set $\Lambda$  of topological dimension 1 containing $\widehat\Lambda_\ell$.
	Observe that $\Omega(G|_\Lambda)$ contains all periodic points in $\Lambda$ and thus all periodic points in $\widehat\Lambda_\ell$. As $\Omega(G|_\Lambda)$ is closed it hence contains $\widehat\Lambda_\ell$.
	By the spectral decomposition theorem, $\Omega(G|_\Lambda)$ splits into finitely many basic sets $\Lambda_1\cup\ldots\cup\Lambda_k$.	
	We then can apply Proposition~\ref{p:p5.3} repeated times to obtain a basic set containing $\Lambda_1\cup\ldots\cup\Lambda_k$ and hence $\widehat\Lambda_\ell$.
\end{proof}

Thus, finally we can prove Theorem~\ref{cor:dense} by giving a recipe to construct large basic sets with prescribed properties.

\begin{proof}[Proof of Theorem~\ref{cor:dense}]
     Let $\{\widehat\Lambda_\ell\}_{\ell\ge1}$ be the sequence of sets constructed in~\eqref{guy}. By Corollary~\ref{cor:allincluded}, for each $\ell$ there is a basic set $\widetilde\Lambda_\ell\supset\widehat\Lambda_\ell$.
     By Proposition~\ref{lem:exist-2} there exists a sequence $\{v_\ell\}_{\ell\ge1}$ of vectors tangent to closed geodesics such that $\chi(v_\ell)\to\underline\chi$. We denote the corresponding periodic orbits by $\cO(v_\ell)$. Analogously, by Lemma~\ref{l:largeperiodic} there exists a sequence $\{w_\ell\}_{\ell\ge1}$ of vectors tangent to closed geodesics such that $\chi(w_\ell)\to\overline\chi$. We denote the corresponding periodic orbits by $\cO(w_\ell)$. Observe that any hyperbolic periodic orbit is locally maximal and hence basic.

    We start by applying Proposition~\ref{p:p5.3} twice to obtain a basic set $\Lambda_1$ containing $\widetilde\Lambda_1\cup\cO(v_1)\cup\cO(w_1)$. Then, again applying Proposition~\ref{p:p5.3} repeated times, for every $\ell\ge2$ there exists a basic set $\Lambda_\ell$ of topological dimension 1 such that
    \[
    \Lambda_\ell\supset
    \Lambda_{\ell-1}\cup
    \widetilde\Lambda_\ell\cup\cO(v_\ell)\cup\cO(w_\ell).
    \]

    Recall our  choice of vectors $v_k$ and $w_k$ and observe that \[\overline\chi(\Lambda_k)\ge \chi(v_k),\chi(w_k)\ge\underline\chi(\Lambda_k).\] This immediately implies~\eqref{e:convergenceexpo}.

    By Lemma~\ref{lem:dense}, the so constructed sequence of basic sets is dense in $T^1M$.
    Finally, given a basic set $\Lambda\subset\cR$, $\Lambda\ne\cR$, then $\Lambda\subset \widehat\Lambda_\ell\subset\widetilde\Lambda_\ell\subset\Lambda_\ell$ for some $\ell\ge1$.
    This proves the theorem.
\end{proof}


\section{Pressure functions and basic sets}\label{s:6}

\subsection{Hyperbolic pressure}

We now introduce three types of pressure of the flow (see~\eqref{d:hypp},~\eqref{d:hypvarp}, and~\eqref{d:hypvarpbas}) and show their equivalence to the topological pressure $P(\varphi)$ introduced above. Given a continuous potential $\varphi\colon T^1M\to\bR$, let
\begin{equation}\label{d:hypp}
P_{\rm hyp}(\varphi)\eqdef \sup_{\Lambda\subset  T^1M} P(\varphi,\Lambda),
\end{equation}
with the supremum taken over all compact ${G}$-invariant hyperbolic sets $\Lambda\subset  T^1M$.
Let us also define
\begin{equation}\label{d:hypvarp}
	P_{\rm hyp\,var}(\varphi)\eqdef
	\sup_{\nu\in\cM_{\rm e}\text{ hyperbolic }}
		P(\varphi,\nu)
\end{equation}
with the supremum taken over all ergodic hyperbolic measures.

We call a measure $\mu\in \cM$ \emph{basic} if it is supported on a basic set. These measures will play a fundamental role in proving Theorem~\ref{mainthm}.
Obviously, every basic measure is hyperbolic. However, a hyperbolic measure may not be basic. For example, the ergodic component $\widetilde m$ of the Liouville measure is hyperbolic, but supported on $T^1M$ and hence basic if and only if the flow is Anosov.
Let
\begin{equation}\label{d:hypvarpbas}
	P_{\rm bas\,var}(\varphi)\eqdef
	\sup_{\nu\in\cM_{\rm e}\text{ basic }}
		P(\varphi,\nu)
\end{equation}
with the supremum taken over all ergodic basic measures.

Let us now describe a way of obtaining basic sets with certain prescribed properties.
The following proposition enables us, in particular, to approximate any hyperbolic ergodic 
measure by a basic one with comparable entropy, pressure, and Lyapunov exponent.

\begin{proposition}[Katok's horseshoes]\label{lem:horse}
	Let $\mu\in\cM_{\rm e}$.
	Let $\varphi\colon T^1M\to \bR$ be continuous. Then for every $\varepsilon>0$ there exists a basic set $\Lambda\subset\cR$ of topological dimension $1$ such that \\[0.2cm]
 	(i) $\,\,h(\Lambda)\ge h(\mu)-\varepsilon$, \\[0.1cm]
	(ii) $\displaystyle P(\varphi,{\Lambda})
		\ge P(\varphi,\mu)
		-\varepsilon$.\\[0.1cm]	
	(iii) $\,\,\chi(\mu)-\varepsilon \le\chi(v)\le\chi(\mu)+\varepsilon\,\,$ for all Lyapunov regular $v\in \Lambda$.
\end{proposition}

Proposition~\ref{lem:horse} is a flow version of a horseshoe approximation by Katok (see~\cite[Supplement S.5]{KatHas:95} and \cite[Theorem 4.1]{Kat:82} for a related result and indications of modifications that are needed in the case of a flow).
A detailed proof can be given by means of nowadays standard methods using adapted Lyapunov metrics
(see for example~\cite[Theorem 2.3]{New:88} for an exposition on systems with vanishing Lyapunov exponents). Note that if $h(\mu)=0$ then the basic set can simply chosen to be a periodic orbit with corresponding properties (see~\cite[Theorem 8]{Kat:84}).

Note that Lemma~\ref{lem:dense} guarantees that, in particular, the above constructed sets $\Lambda_\ell$ eventually contain any basic set as is provided in Proposition~\ref{lem:horse}.

\begin{theorem}\label{theorem:3}
	Let $\varphi\colon T^1M\to\bR$ be a continuous potential.
	If $\cH\ne\emptyset$ then assume that $\varphi|_\cH$ is constant. Then
	\[
		P_{\rm hyp}(\varphi)
		= P_{\rm bas\,var}(\varphi)
		= P_{\rm hyp\,var}(\varphi)
		= \sup_{\ell\ge 1}P(\varphi,\Lambda_\ell)
		= \lim_{\ell\to\infty}P(\varphi,\Lambda_\ell)
        		= P(\varphi).
        \]
\end{theorem}

\begin{proof}
If $\cH=\emptyset$ then $T^1M$ is hyperbolic for $G$ and the claimed property immediately follows.
Let us hence assume that $\cH\ne\emptyset$. Without loss of generality we can assume that $\varphi=0$ on $\cH$. Indeed, otherwise let $c\eqdef \varphi|_\cH$ and replace $\varphi$ by $\varphi-c$ and observe that any of the above pressure functions satisfies $P(\varphi-c)=P(\varphi)-c$.

Given any hyperbolic set $\Lambda\subset T^1M$, by the variational principle applied to $G|_\Lambda$ and Proposition~\ref{lem:horse} we obtain
\[
	P(\varphi,\Lambda)
	= \sup_{\nu\in\cM_{\rm e}(\Lambda)}P(\varphi,\nu)
	\le  \sup_{\nu\in\cM_{\rm e}(\Lambda)\text{ basic}}P(\varphi,\nu)
	\le  P_{\rm bas\, var}(\varphi).
\]
With this property the following inequalities
\[
	P_{\rm hyp}(\varphi)
		\le P_{\rm bas\,var}(\varphi)
		\le P_{\rm hyp\,var}(\varphi)
		\le P(\varphi)\] are easy to verify.

Let us prove the opposite inequality $P(\varphi)\le P_{\rm hyp}(\varphi)$.
The hypothesis that $\varphi|_\cH=0$ ensures that $P(\varphi) \geq 0$.
Indeed, if $\mu$ is a measure supported on $\cH$, we have $\int \varphi\,d\mu = 0$ and $h(g^1,\mu) = 0$ by the Ruelle inequality, and then the variational principle ensures that $P(\varphi) \geq 0$. We distinguish two cases.\\[0.1cm]
\textbf{Case 1:} If $P(\varphi) = 0$, use the method of the proof of Lemma~\ref{lem:exist-1} to find a closed orbit $\cO$ in $\cR$ that most of its time stays close to $\cH$. The invariant measure $\nu$ that is supported on such an orbit is hyperbolic and satisfies
\[
	\int\varphi\,d\nu\ge -\varepsilon
\]	
for some small $\varepsilon > 0$. Hence $P_{\rm hyp}(\varphi)\ge P(\varphi,\cO)\ge -\varepsilon$. As $\varepsilon$ can be made arbitrarily small, we obtain $P_{\rm hyp}(\varphi)\ge 0=P(\varphi)$.
\\[0.1cm]
\textbf{Case 2:}
Now suppose that $P(\varphi) > 0$. Let $\mu$ be an equilibrium state for $\varphi$. Without loss of generality we can assume that $\mu$ is ergodic. Indeed, every ergodic component is also an equilibrium state. 
By Proposition~\ref{lem:horse} (iii), for every $\varepsilon>0$ there is a basic set $\Lambda\subset\cR$ such that
\[
	P_{\rm hyp}(\varphi)
	\ge P(\varphi,{\Lambda})
	\ge P(\varphi,\mu) 
		-\varepsilon=P(\varphi)-\varepsilon.
\]	
As $\varepsilon$ can be made arbitrarily small, we obtain $P_{\rm hyp}(\varphi)\ge  P(\varphi)$.

Since the increasing family of hyperbolic sets $\Lambda_\ell$ eventually contains any hyperbolic invariant set, we have
\[
	P(\varphi) \ge
 	\sup_{\ell \geq 0} P(\varphi,{\Lambda_\ell})
 	= \lim_{\ell \to \infty}P(\varphi,{\Lambda_\ell}).
\]
Together with Proposition~\ref{lem:horse} we finally conclude $\sup_\ell P(\varphi,{\Lambda_\ell})\ge P_{\rm hyp}(\varphi)$.

This proves the theorem.
\end{proof}

\subsection{Approximating by the pressure on basic sets}
We now come to a simple but  crucial proposition. Let $\Lambda_\ell$ be the sequence of basic sets in Theorem~\ref{cor:dense}. It will be convenient to write $\cP_\ell$, $\cD_\ell$, $\cE_\ell$, and $\cF_{\ell,\alpha}$ instead of $\cP_{\Lambda_\ell}$, $\cD_{\Lambda_\ell}$, $\cE_{\Lambda_\ell}$, and $\cF_{{\Lambda_\ell,\alpha}}$.

\begin{proposition}\label{unifconvprop}
	$\cP_\ell$ converges to $\cP$ uniformly on compact intervals.
\end{proposition}

\begin{proof}The functions $\cP_\ell$ and $\cP$ are continuous
because they are convex. The sequence $\cP_{\Lambda_\ell}$ is monotone
because the sets $\Lambda_\ell$ are nested. Theorem~\ref{theorem:3} shows
that $\cP_\ell$ converges to $\cP$ pointwise.
Now apply Dini's theorem.
\end{proof}

It follows from Theorem~\ref{cor:dense} that for each $\alpha\in(\underline\chi,\overline\chi)$ the linear function $\cF_{\ell,\alpha}$, whose graph is the supporting line to $\cP_\ell$ with slope $-\alpha$,  is defined for all large enough~$\ell$.

\begin{proposition} \label{lineconv}
	For each $\alpha \in(\underline\chi,\overline\chi)$ we have
	\[
		\cF_\alpha = \lim_{\ell \to \infty} \cF_{\ell,\alpha}.
	\]	
\end{proposition}

  \begin{proof} It follows easily from the monotone convergence established in the proof of  Proposition~\ref{unifconvprop} that the sequence $\cF_{\ell,\alpha}$ is nondecreasing and $$\lim_{\ell \to \infty} \cF_{\ell,\alpha} \leq \cF_\alpha.$$
  We now show that for any $\varepsilon > 0$ we have  $\cF_{\ell,\alpha} \geq \cF_\alpha - \varepsilon$ for all large enough $\ell$. It will suffice to show that $\cP_\ell \geq \cF_\alpha - \varepsilon$ for all large enough $\ell$. Let
  $$
  J_\ell = \big\{q\colon \cP_\ell(q) \leq \cF_\alpha(q) - \varepsilon\big\}.
  $$
  Since $\cP_\ell$ is convex and $\cF_\alpha$ is linear, $J_\ell$ is a closed interval. By
 Theorem~\ref{cor:dense}, we can choose $\ell_0$ such that $\alpha \in (\underline\chi(\Lambda_{\ell_0}), \overline\chi(\Lambda_{\ell_0}))$. It then follows from (\ref{PL2}) of Proposition~\ref{convpropforbasic} that $J_{\ell_0}$ is bounded. For $\ell \geq \ell_0$, we have $J_\ell \subset J_{\ell_0}$ since $\cP_\ell \geq \cP_{\ell_0}$. Hence $\cP_\ell \geq \cF_\alpha - \varepsilon$ outside $J_{\ell_0}$ for all $\ell \geq \ell_0$.

 On the other hand, since $\cP_\ell \to \cP$ uniformly on compact sets by Proposition~\ref{unifconvprop}, we have $\cP_\ell \geq \cP - \varepsilon$ on $J_{\ell_0}$ for all large enough $\ell$.  But $\cP \geq \cF_\alpha$, so we obtain $\cP_\ell \geq   \cF_\alpha - \varepsilon$ on $J_{\ell_0}$ for all large enough $\ell$.
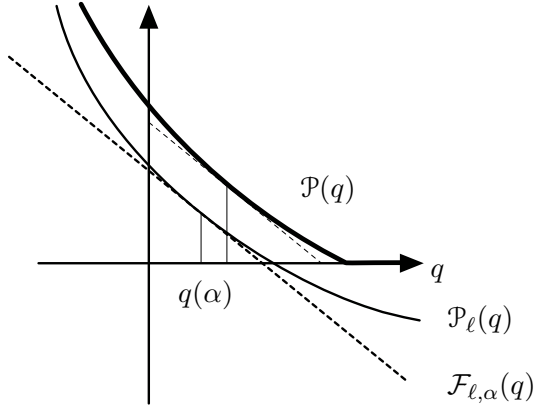
\begin{figure}
\begin{minipage}[t]{\linewidth}
\centering
	\begin{overpic}[scale=.90,
  ]{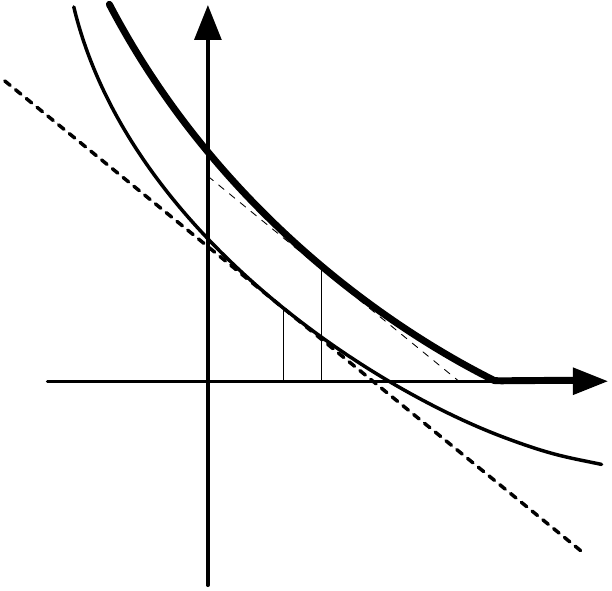}
	\put(101,31){$q$}	
	\put(70,50){$\cP(q)$}	
	\put(41,25){$q(\alpha)$}	
	\put(105,3){$\cF_{\ell,\alpha}(q)$} 
	\put(105,19){$\cP_{\ell}(q)$}
\end{overpic}
\end{minipage}
\caption{Approximation of pressure by the pressure on a basic sets $\Lambda_\ell\subset T^1M$}\label{fig.map4}
\end{figure}
 \end{proof}

Let us now provide the proof of Proposition~\ref{nonemptyprop} stating that $\cL(\alpha) \neq \emptyset$ for $\alpha \in [\underline\chi,\overline\chi]$.

\begin{proof}[Proof of Proposition~\ref{nonemptyprop}]
	Given a basic set $\Lambda\subset T^1M$, with~\cite[Corollary 5.1]{PesSad:01} the range of Lyapunov exponents of Lyapunov regular vectors in $\Lambda$ is  the closed interval $[\underline\chi(\Lambda),\overline\chi(\Lambda)]$.
	For every $\alpha\in(\underline\chi,\overline\chi)$, by Theorem~\ref{cor:dense}  there exists a basic set $\Lambda_\ell$, $\ell\ge1$, such that $\alpha\in(\underline\chi(\Lambda_\ell),\overline\chi(\Lambda_\ell))$ and hence $\cL(\alpha)\supset\cL(\alpha)\cap\Lambda_\ell\ne\emptyset$.
	Finally $\cL(\underline\chi)$, $\cL(\overline\chi)\ne\emptyset$ by Corollary~\ref{expmeas}.
\end{proof}

We can now prove Theorem~\ref{mainthm} stating that $\dim_{\rm H}\cL(\alpha) \geq 1+ 2\,\cD(\alpha)$ and $h(\cL(\alpha)) \geq \cE(\alpha)$ for $\alpha \in (\underline\chi,\overline\chi)$.

\begin{proof}[Proof of Theorem~\ref{mainthm}]
It follows from Proposition~\ref{lineconv} that $\cD_\ell \to \cD$ and $\cE_\ell \to \cE$ as $\ell \to \infty$. Since  $\dim_{\rm H}(\cL(\alpha) \cap \Lambda_\ell)$ and $h(\cL(\alpha) \cap \Lambda_\ell)$ are lower bounds for
  $\dim_{\rm H}\cL(\alpha)$ and $h(\cL(\alpha))$, it follows immediately from Proposition~\ref{unifconvprop} that
  $$
  \text{$\dim_{\rm H}\cL(\alpha) \geq \cD(\alpha)$\quad and \quad $h(\cL(\alpha)) \geq \cE(\alpha)$.}
  $$
  This proves the theorem.
\end{proof}

We can also provide the proof of Theorem~\ref{lem:lowe} stating that the entropies $h(\Lambda)$ of basic sets $\Lambda\subset T^1M$ are dense in $[0,h]$.

\begin{proof}[Proof of Theorem~\ref{lem:lowe}]
We distinguish two cases:\\[0.1cm]
\smallskip
\textbf{Case $\cH\ne\emptyset$:} By the variational principle for entropy, for any small $\delta>0$ there exists $\mu\in\cM_{\rm e}$ such that $h(\mu)\ge h-\delta/2$, which in particular implies  that $\mu$ is hyperbolic. By Proposition~\ref{lem:horse}, there exists a basic set $\Lambda$ such that $h(\Lambda)\ge h(\mu)-\delta/2\ge h-\delta$.  By Ruelle's inequality, in particular we have $\overline\chi(\Lambda)\ge h-\delta$.
  By Proposition~\ref{lem:exist-1}, there exists a closed orbit $\cO(v)$ through a vector  $v\in\cR$ with Lyapunov exponent $0<\chi(v)<\delta$. By Proposition~\ref{p:p5.3}, there exists a basic set $\widehat\Lambda$ that contains $\Lambda\cup \cO(v)$ and hence satisfies
  \[
	  \underline\chi(\widehat\Lambda)
  	\in (0, \delta),\quad
	\overline\chi(\widehat\Lambda)\ge h-\delta,\quad\text{ and }\quad
	h(\widehat\Lambda)\ge h-\delta.
  \]
In particular (compare Section~\ref{sec:3.1basicsets}), we conclude that the range of the entropy spectrum $\alpha\mapsto h(\cL(\alpha)\cap\widehat\Lambda)$ contains the interval $[\delta,h(\Lambda)]$ and for any value $h'\in (\delta,h-\delta)$
there exist a number $q$ and the equilibrium state $\mu_q$ of $q \varphi^u$ with respect to ${G}|_{\widehat\Lambda}$ satisfying
\[
	h_{\mu_q} = P(q \varphi^u,\widehat\Lambda\,) - q\,\chi(\mu_q) = h'.
\]
Now applying Proposition~\ref{lem:horse} one more time finishes the proof.\\[0.1cm]
\textbf{Case $\cH=\emptyset$:} Choose a nonpositive H\"older continuous function $\varphi\colon T^1M \to \bR$ that is $0$ on one closed orbit $\cO$ and negative elsewhere and consider the equilibrium states for the functions $q\,\phi$ for $q \geq 0$. The function $q\mapsto P(q\varphi,\Lambda)$ is real analytic and strictly convex function and there is a unique equilibrium state for each $q$. The entropy of this equilibrium state is equal $h(\Lambda)$ for $q = 0$ and decreases as $q$ increases. Note that the measure $\mu$ that is supported on the periodic orbit $\cO$ satisfies $P(q\varphi,\cO)=0$ for all $q$. Hence the pressure is nonnegative for all $q$. For large $q$, the integral of $\varphi$ with respect to equilibrium measure for $q\varphi$ must be small. Weak star compactness and the fact that the closed orbit where $\phi$ vanishes supports the only measure for which the integral of $\varphi$ is $0$ mean that the equilibrium state converges to this measure as $q \to \infty$. Then upper semicontinuity of the measure entropy (which follows from expansivity) forces the entropies of the equilibrium states to approach $0$.
\end{proof}

\bibliographystyle{amsplain}

\end{document}